\documentclass[12pt]{amsart}
\usepackage{amssymb,amsmath,amsthm,mathtools,amscd,graphicx}
\usepackage{enumitem}
\usepackage[margin=1in]{geometry}
\usepackage{mleftright}
\mleftright

\usepackage{array}

\usepackage{nccmath}

\def\Z{\mathbb{Z}}
\def\Q{\mathbb{Q}}
\def\R{\mathbb{R}}
\def\H{\mathbb{H}}

\def\C{\mathbb{C}}
\def\B{\mathcal{B}}
\def\a{\alpha}
\def\b{\beta}

\def\G{\Gamma}
\def\T{\mathcal{T}}
\def\half{\tfrac{1}{2}}

\newcommand{\til}{\widetilde}

\def\PGL{{\rm PGL}}

\newcommand{\pMatrix}[4]{\left(\begin{matrix}#1 & #2 \\ #3 & #4\end{matrix}\right)}

\renewcommand{\pmatrix}[4]{\left(\begin{smallmatrix}#1 & #2 \\ #3 & #4\end{smallmatrix}\right)}
\renewcommand{\bar}[1]{\overline{#1}}

\def \ep {\varepsilon}

\newcommand{\sC}{\mathcal{C}}

\newtheorem{theorem}{Theorem}
\newtheorem{lemma}{Lemma}

\newtheorem{proposition}{Proposition}

\theoremstyle{remark}

\numberwithin{equation}{section}

\mathtoolsset{showonlyrefs}

\title{ Markov spectra for modular billiards }
\date{\today}

\author{Nickolas Andersen}
\address{UCLA Mathematics Department,
Box 951555, Los Angeles, CA 90095-1555} \email{nandersen@math.ucla.edu}

\author{William Duke}
\address{UCLA Mathematics Department,
Box 951555, Los Angeles, CA 90095-1555} \email{wdduke@ucla.edu}
\thanks{Supported by NSF grant DMS 1701638.}
 \dedicatory{In memory of Harvey Cohn (1923--2014)}

\begin{document}

\begin{abstract}
We introduce some analogues of the Markov spectrum defined in terms of modular billiards
and consider the problem of characterizing  that part of the spectrum below the lowest limit point.
\end{abstract}

\maketitle

\section{Introduction}

The abstract  triangle group usually denoted by $\Delta(2,3,\infty)$ is generated by  $A,B,C$
subject to the relations $A^2=B^2=C^2=(AB)^2=(AC)^3=1$.
 The extended modular group  $\G=\PGL(2,\Z)$ gives a faithful representation of this triangle group
when we make the identifications:
\begin{equation}\label{ABC}
A=\pm\pMatrix{0}{1}{1}{0},\;\;B=\pm\pMatrix{-1}{0}{0}{1},\;\;C=\pm\pMatrix{-1}{1}{0}{1}.
\end{equation}
The usual modular group $\mathrm{PSL}(2,\Z)$ is the subgroup of index 2 
consisting of all matrices in $\G$ with determinant one.
 
Let $\H$ be the upper half-plane with its hyperbolic metric given by $ds=\frac{|dz|}{y}$.
It is well known that  $M=\pm \pmatrix{a}{b}{c}{d}\in \G$ with  $\det M =1$ acts as an orientation preserving isometry of $\H$
through 
\begin{equation}\label{act}
z \mapsto M(z)= \frac{az+b}{cz+d},\end{equation}
 while  when $\det M =-1$  it acts through $z \mapsto   M (\bar{z})$ as 
an orientation reversing isometry. 
The generators $A,B,C$ give reflections across the unit circle, the $y$-axis, and the line $x=\frac{1}{2}$, respectively; $\G$ acts as a reflection group.
 A convenient fundamental domain for $\G$ is the solid hyperbolic triangle 
 \begin{equation}\label{T}
 \mathcal{T}=\{z\in \H; 0 \leq \mathrm{Re} \,z \leq \half ,  |z|\geq 1\},
 \end{equation}
whose sides are fixed by the generating reflections  and which is the shaded region depicted in Figure \ref{fig1}.
\begin{figure}[ht]
   \includegraphics[height=1.5in]{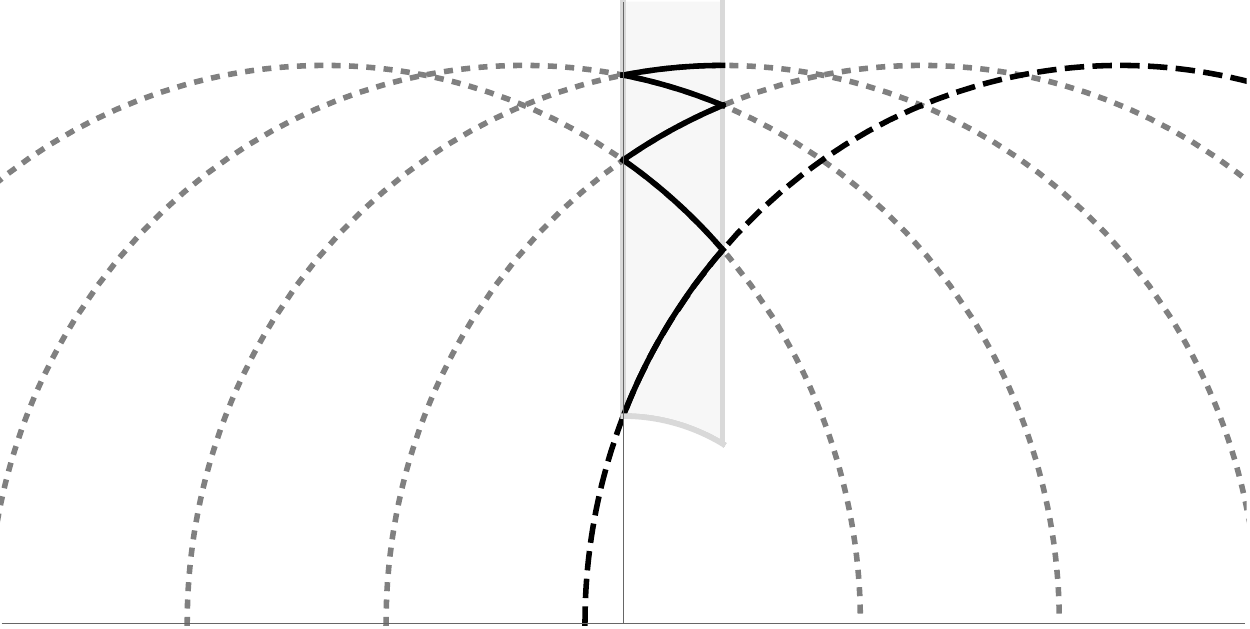}
         \caption{The modular triangle $\T$ and a modular billiard  }
    \label{fig1}
\end{figure}

 Let $S$  be an oriented geodesic in $\H$. 
 Thus $S$ is given either by a directed vertical half-line or a directed semi-circle that is orthogonal to $\R$
 and is uniquely determined by ordering its endpoints, say  $\a,\b$, which are distinct elements of $\R\cup \{\infty\}.$
 More generally, for $z_1,z_2\in \H\cup \R\cup \{\infty\}$
 let $\langle z_1,z_2\rangle$ denote the geodesic segment connecting $z_1$ to $z_2$. 
Hence we may write $S=\langle\a,\b\rangle$.
 
 The set of all geodesics splits into orbits $\G S$ under the action of $\G$, where $S$ is a fixed geodesic.
Let $\B$ denote the set of distinct directed geodesic segments in $\T$ of an orbit $\G S$.
We will refer to $\B$ as the trajectory of a modular billiard, but usually call it simply a  {\it  modular billiard}.
We will say that $\B$ is induced by $S$ for any $S$ in the orbit.
Note that $\B$  can be thought of as the path of a point   acting like a billiard ball bouncing off the sides
of $\T$, with  well-defined  bounces from the corners of $\T$, which are at 
\[z=i \;\;\;\mathrm{and}\;\;\;\;z=\rho=\half+\tfrac{\sqrt{-3}}{2}.\]

Suppose that $\B$ is induced by $\langle\a,\b\rangle$. Define its {\it reversal}  $\B^*$ to be the billiard induced by $\langle\b,\a\rangle$.
We say the billiard $\B$ is {\it non-orientable} if $\B =\B^*$,
 {\it orientable} otherwise.   If $\B$  contains a vertical segment we say it is  {\it improper}, otherwise {\it proper}. If  the total hyperbolic length of the segments in $\B$ is finite, we call the billiard {\it periodic}. Clearly a periodic billiard is proper.
The billiard  illustrated in Figure \ref{fig1} is  non-orientable  and  periodic.

 The simplest modular billiard, which we will denote $\mathcal{C}_0$, is that induced by 
 the imaginary axis $\langle0,\infty\rangle.$  It covers the segment connecting $i$ to infinity.  The  billiard induced by $\langle\frac{1}{2},\infty\rangle$,  denoted $\mathcal{C}_{\frac{1}{2}},$ 
covers the rest of the boundary of $\T.$  Both $\mathcal{C}_0$ and $\mathcal{C}_{\frac{1}{2}}$ are improper  and non-orientable.

In a prescient article of 1924, Artin \cite{Ar} observed that  properties of continued fractions imply that a generic modular billiard is dense in $\T$.
\begin{figure}[ht]
  \includegraphics[trim=.5em 0 0 .7em, clip, height=1.2in]{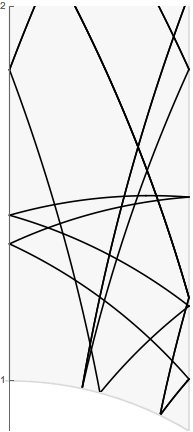} \qquad \quad  
  \includegraphics[trim=.5em 0 0 .7em, clip, height=1.2in]{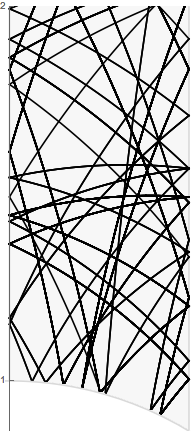} \qquad \quad 
  \includegraphics[trim=.5em 0 0 .7em, clip, height=1.2in]{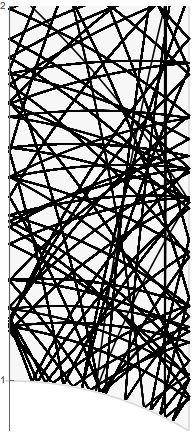}
         \caption{Increasing segments  of a billiard exhibiting generic behavior}
    \label{fig:generic}
\end{figure}
On the other hand,   the behavior of a non-generic billiard is  subtle and can be  quite interesting arithmetically.
For instance, a modular billiard $\B$ has a maximal height, possibly infinite, defined to be the supremum of imaginary parts of points on $\B$. Let $\lambda_\infty(\B)$ be {\it twice} this maximal height. Consider 
the set 
\[\mathcal{M}_\infty=\{\lambda_\infty(\B); \B \;\;\text{is a modular billiard}\}.\]
This is the {\it Markov spectrum}, which is usually defined (equivalently) in terms of the minima of indefinite binary quadratic forms.
The Markov numbers are 
those positive integers $p$ for which there  are $q,r\in \Z^+$ such that
 \begin{equation}\label{mareq}
  p^2 +q^2+r^2 =3pqr.
  \end{equation}
  These may be ordered into an infinite increasing sequence  whose $n^{th}$ term is denoted by $p_n$:
  \[\{1, 2, 5, 13, 29, 34, \dots ,p_n,\dots\}.\]
  The following result is a consequence of  the fundamental work of  A.~A.~Markov  \cite{Mark}:
  \begin{theorem}\label{mark}
  For any fixed $\kappa<3$ there are only finitely many modular billiards $\B$ with $\lambda_\infty(\B)<\kappa.$
The points in $ \mathcal{M}_\infty$ less than $3$ are given by the sequence
 \[
\left\{\sqrt{5},\sqrt{8},\tfrac{\sqrt{221}}{5},\tfrac{\sqrt{1517}}{13},\tfrac{\sqrt{7565}}{29},\dots,\tfrac{ \sqrt{9p_n^2-4}}{p_n},\dots. \right\},
 \]
which  is monotone increasing to the limit $3\in \mathcal{M}_\infty$.
 \end{theorem} 
  \begin{figure}[ht]
  \includegraphics[height=1.in]{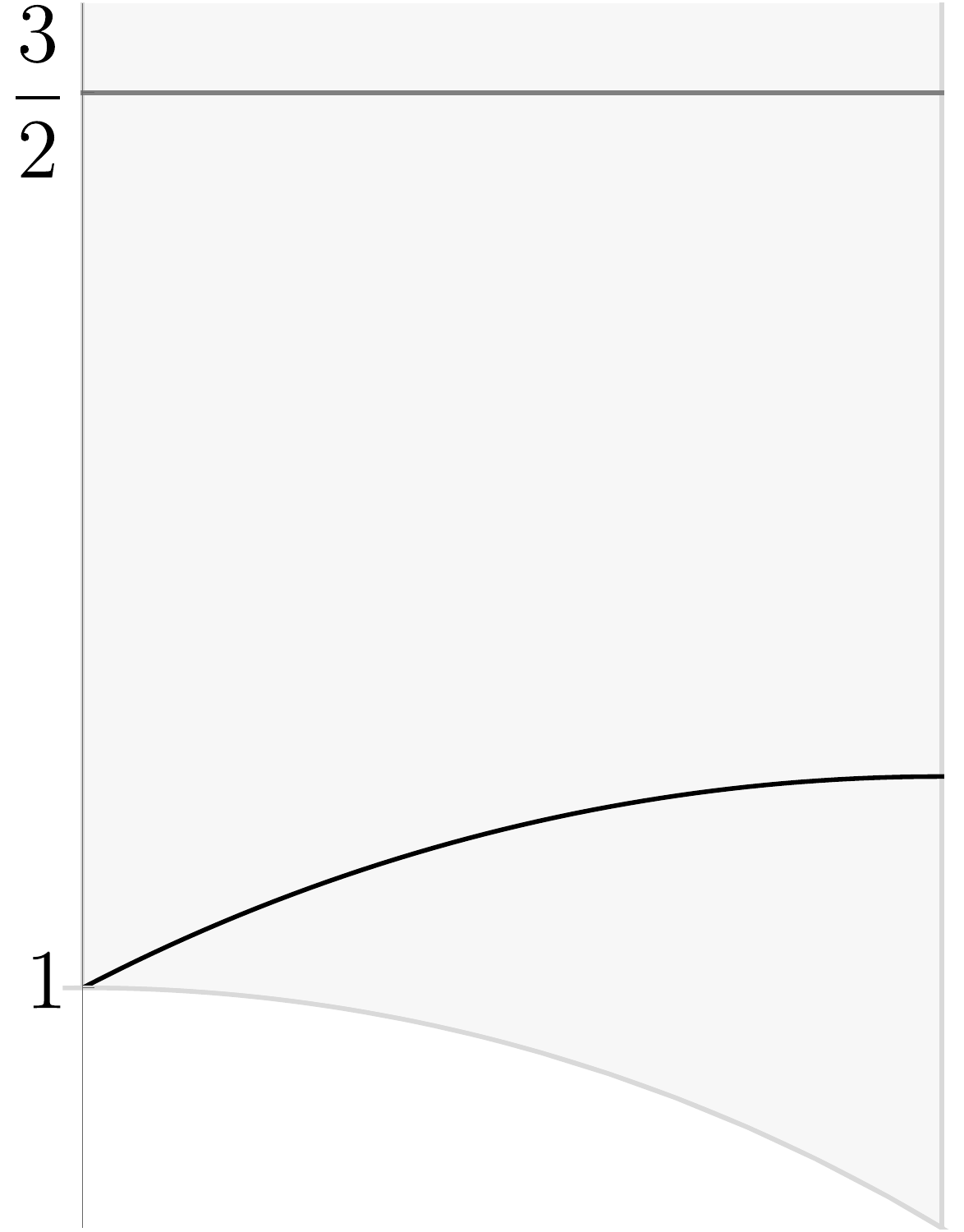}\;\;\;\;\;\; 
  \includegraphics[height=1.in]{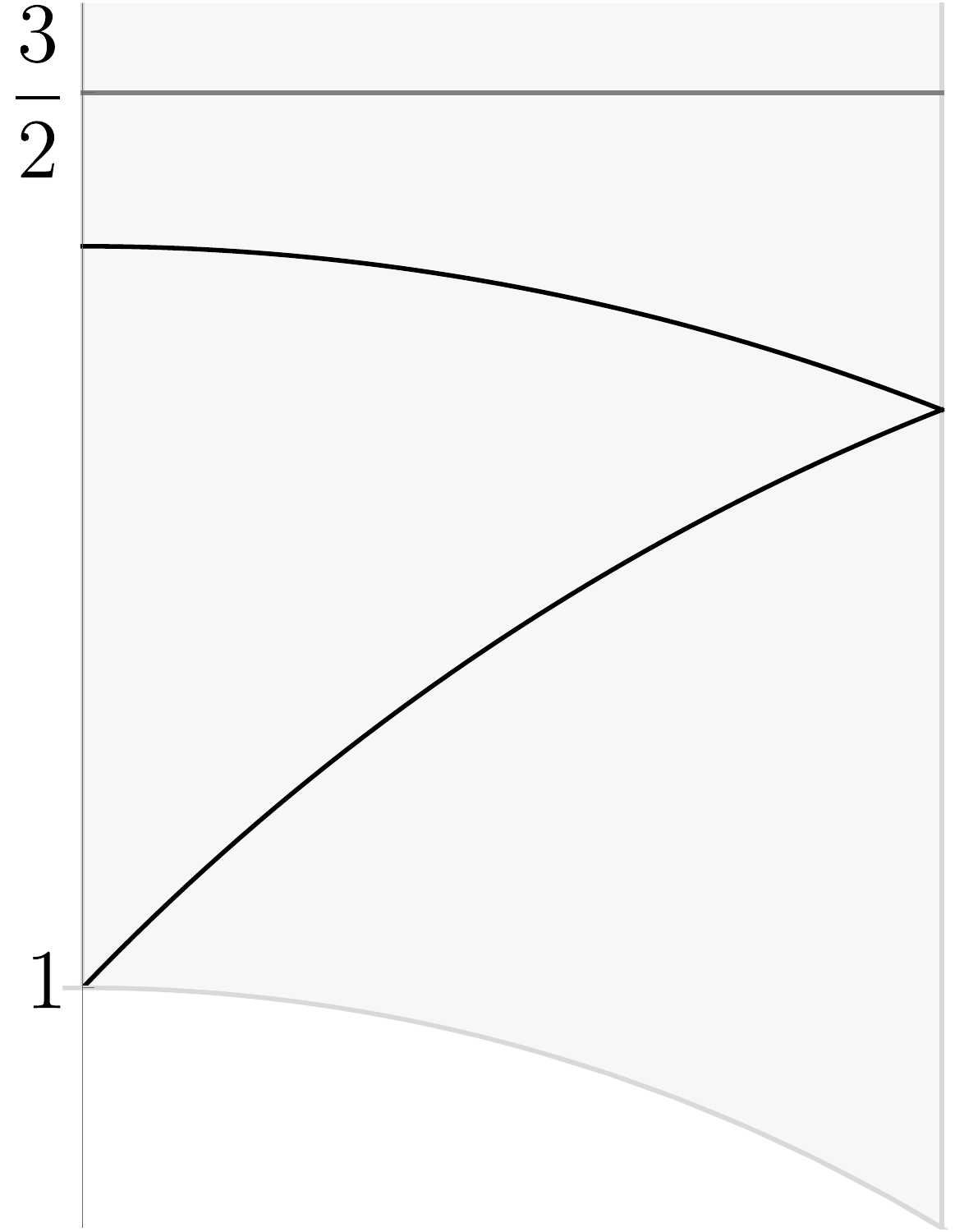}      \;\;\;\;\;\; 
  \includegraphics[height=1.in]{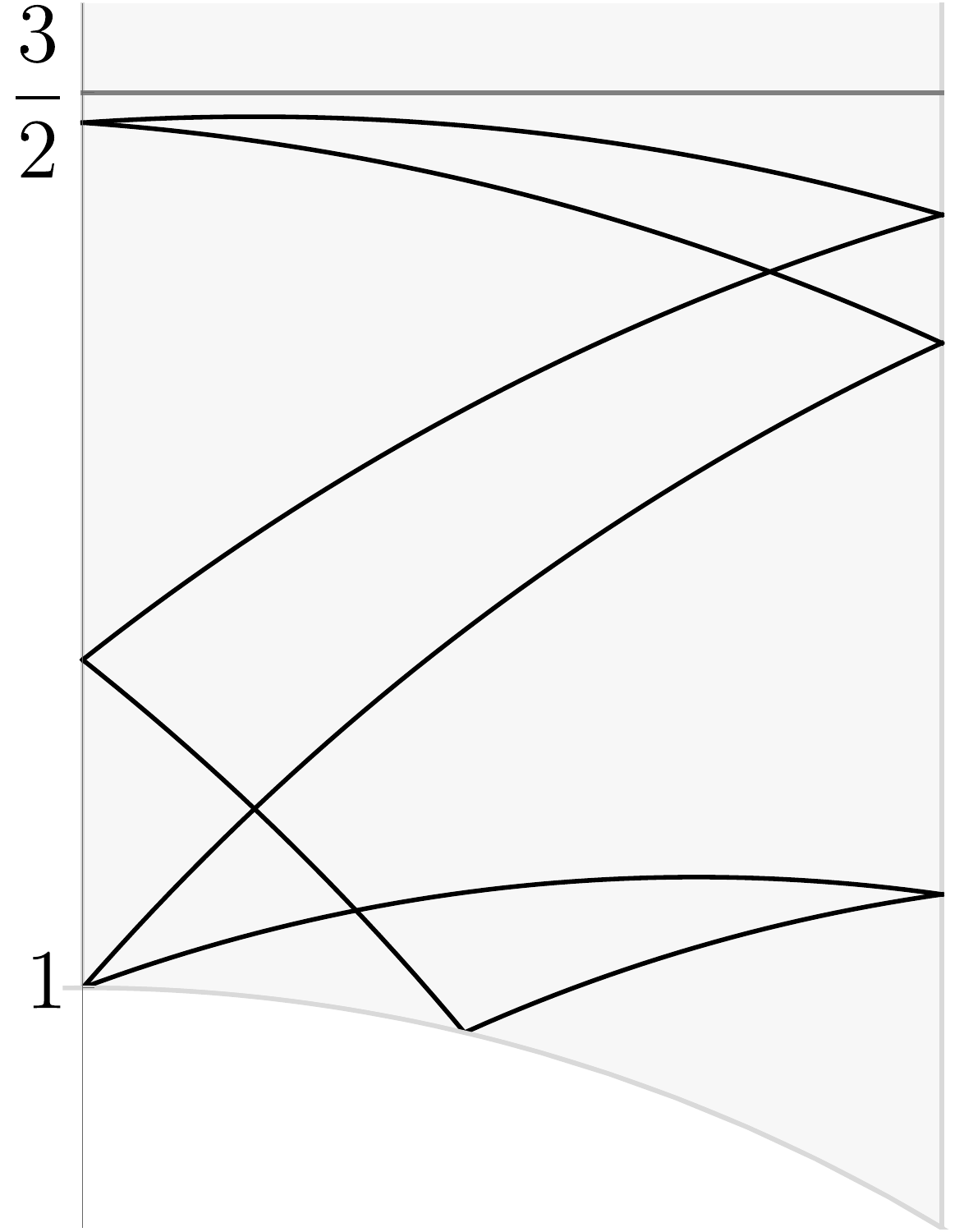}    \;\;\;\;\;\; 
  \includegraphics[height=1.in]{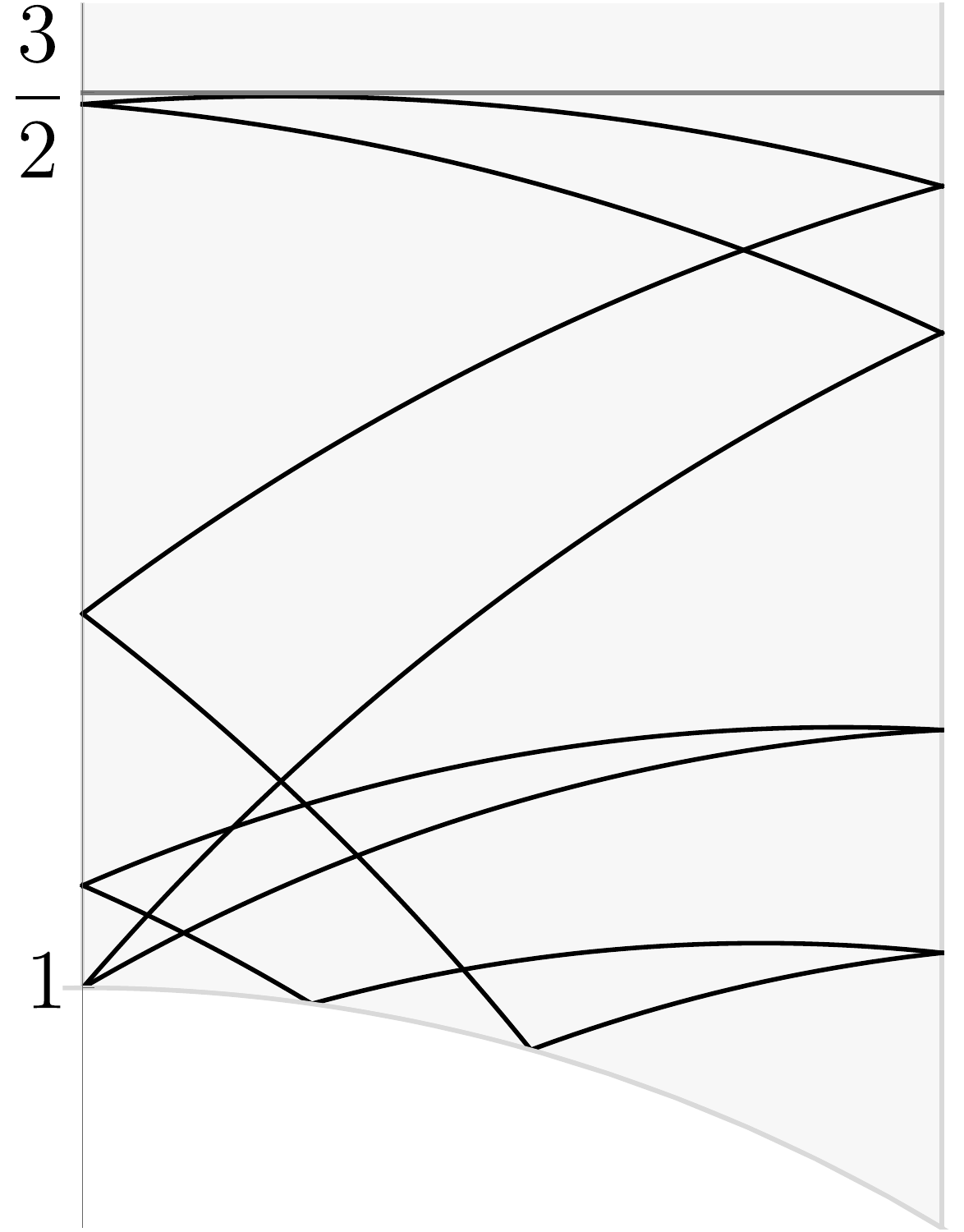}        
     \caption{Billiards associated to the points $\sqrt{5},\sqrt{8},\tfrac{\sqrt{221}}{5},\tfrac{\sqrt{1517}}{13}\in  \mathcal{M}_\infty$}
    \label{fig:markov}
\end{figure}

 It is also known that each of the points $<3$ in $\mathcal{M}_\infty$ is actually attained by a non-orientable periodic  billiard (see Theorem~75 of \cite{Di2})
 and it was conjectured by Frobenius \cite{Fr}, but is still open, that 
 the multiplicity of each of these points is one, meaning that the associated billiard is unique.
The part of the  Markov spectrum that is $>3$ is less understood but has been the subject of much research (see \cite{Ma},\cite{CF}).  It is not hard to show  that any open interval around $3$ contains uncountably many points of  $\mathcal{M}_\infty$ and that $\mathcal{M}_\infty$ is closed, but there are few completely definitive results known.
Building on pioneering work of  Hall~\cite{Ha},   Freiman~\cite{Fre} obtained one such result.  He showed  that 
  $[\mu,\infty)\subset \mathcal{M}_\infty$, where \[\mu=\tfrac{2221564096+283748\sqrt{462}}{491993569}=4.527829566\dots,\] with
no smaller value of $\mu$ being possible. 

We remark that the seminal work of Harvey Cohn, beginning with his 1955 paper \cite{Co}, revealed among other things a completely unexpected relation between the Markov spectrum
and the length spectrum of {\it simple} closed geodesics on the modular torus. His work has had a lasting impact on the study of simple closed geodesics
on Riemann surfaces.  It has also led to a better understanding of the Markov spectrum itself. See \cite{Co2} for a  summary  of some of his contributions.

  The value $\lambda_\infty(\B)^{-1}$  may be thought of as a measure of how close the billiard $\B$ gets to the corner of $\T$ at the cusp $i \infty.$
It is natural to ask how close a modular billiard must get
 to each of the  other corners $i$ and $\rho$  of $\T$.
 By the  distance of a billiard from a point $z\in \T$, denoted by $\delta_z(\B)$, we mean the infimum of the hyperbolic distance between points on the billiard and $z$. 
Let   
 \begin{equation}\label{mar2}
\lambda_z(\B)= (\sinh \delta_z(\B))^{-1}.
\end{equation}
A natural analogue of the Markov spectrum is 
\begin{equation}\label{genmar}
\mathcal{M}_z=\{\lambda_z(\B); \B \;\;\text{is a modular billiard}\}
\end{equation}
for a fixed $z \in \T.$ 

In this paper we will give results about $\mathcal{M}_\rho$ and  $\mathcal{M}_i$ that correspond to Markov's  for $\mathcal{M}_\infty$.
The result for $z=\rho$ is quite easy to prove. 
\begin{theorem}\label{rhothm}
 The smallest value in  $ \mathcal{M}_\rho$ is  $\sqrt{3}$, which is attained by $\mathcal{C}_{0}$.  
 The value $\sqrt{3}$ is a limit point  of $ \mathcal{M}_\rho$.
  \end{theorem} 

 The result for $z=i$ is  deeper and most of this paper is devoted to its proof.
\begin{theorem}\label{main}
  The three smallest values in $ \mathcal{M}_i$ are
  \[
\{\half\sqrt{21},\;\;\tfrac{2}{3}\sqrt{14},\;\;\tfrac{1}{3} (3+\sqrt{21})\}=\{2.29129...,\;2.49444...,\;2.52753...\}.
\] 
These three values are attained, respectively, by unique  billiards $\sC_1,\sC_2$ and $\sC_3$, each proper and  non-orientable.  Here  $\sC_1$ and $\sC_2$ are periodic billiards, while $\sC_3$ is not periodic.  The value $\tfrac{1}{3} (3+\sqrt{21})$ is a limit point of $ \mathcal{M}_i$.
\end{theorem}

    \begin{figure}[ht]
  \includegraphics[height=2.2in]{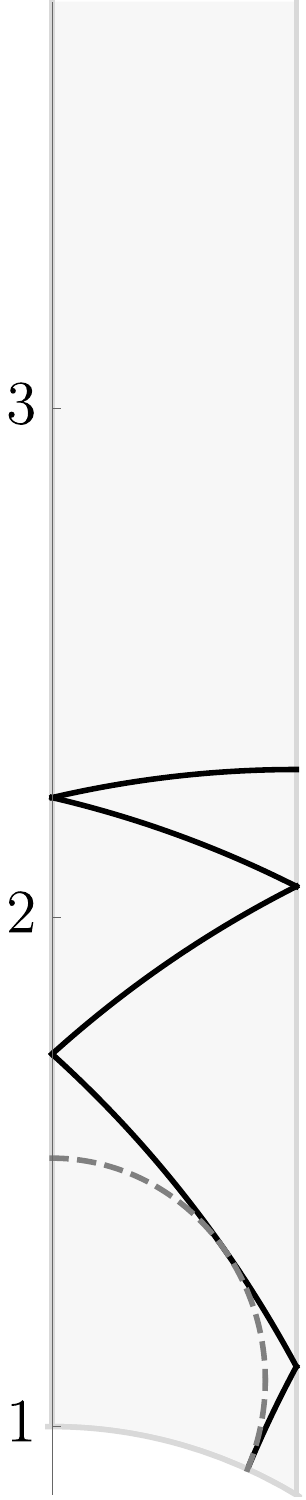}\qquad \qquad 
  \includegraphics[height=2.2in]{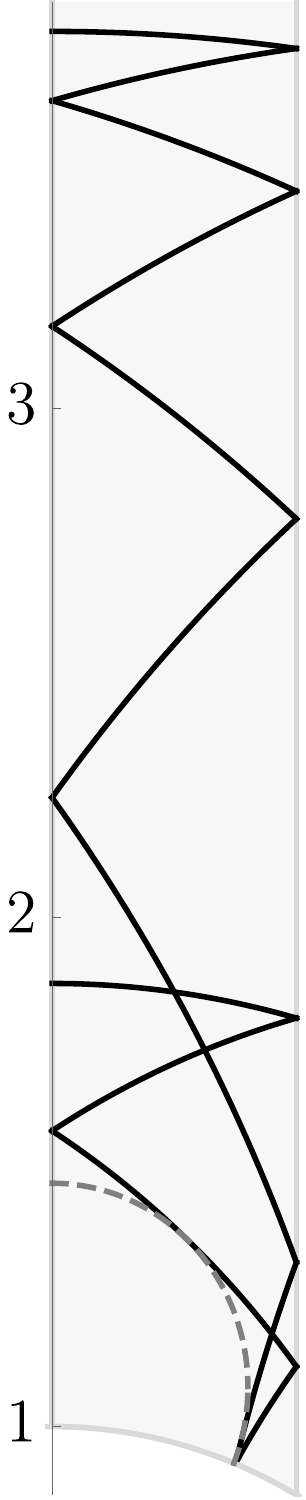}\qquad \qquad 
  \includegraphics[height=2.2in]{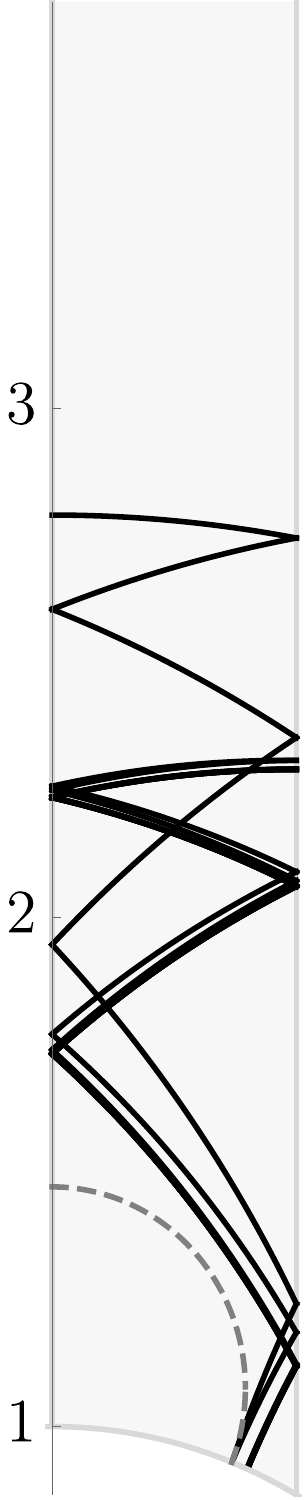}
         \caption{Billiards $\sC_1,\sC_2$ and $\sC_3$
          }
    \label{fig:C123}
\end{figure}

Explicitly, $\sC_1$ is induced by the geodesic $\langle \frac{1}{2}(1-\sqrt{21}),\frac{1}{2}(1+\sqrt{21})\rangle$, $\sC_2$ is induced by the geodesic $\langle \frac{1}{2}(2-\sqrt{14}),\frac{1}{2}(2+\sqrt{14})\rangle$ and $\sC_3$ is induced by the geodesic $\langle \frac{1}{2}(3-\sqrt{21}),\frac{1}{2}(5+\sqrt{21})\rangle$. 

In both cases the rest of the spectrum  invites investigation.
It is also of interest to consider the Markov spectrum $ \mathcal{M}_z$ for other points in $\T$,
in particular CM points.
In addition to distances from a fixed point, 
there are  other geometric quantities associated to non-generic modular billiards 
whose sets of values define  Markov-type spectra.
The purpose of this paper is to initiate a study of these generalizations by concentrating on  the simplest and most natural  examples and giving the analogues of Markov's results for them.

   In the next section we give a geometric interpretation of  Theorems \ref{mark}--\ref{main} in terms of the packing of discs in tessellations
 formed  by geodesic segments and prove the first statement of Theorem \ref{rhothm}.   In \S \ref{bqf} we recall the connection between modular billiards and real indefinite binary quadratic forms and then 
in \S \ref{qqf} give a formula for the hyperbolic distance between a billiard and a point.  This formula is  written  in terms of the minimum of an indefinite 
   quaternary quadratic form and is used to complete the proof of Theorem \ref{rhothm}. In \S \ref{proper} we introduce reduced forms and express $\lambda_i(\B)$ in terms of them.
   Then  we give in \S \ref{seq} the correspondence between proper modular billiards and doubly-infinite sequences of positive integers
 that connects billiards to simple continued fractions. This connection is exploited in \S \ref{seqspec}, \S \ref{ex} and \S \ref{final} to complete the proof of  Theorem \ref{main}. 

{\footnotesize
 \subsection*{Acknowledgement:} The second author thanks Alex Kontorovich for some enlightening discussions
 on the topics of this paper.
}

\section{Packing discs in hyperbolic tessellations}

Elementary geometric considerations provide 
some useful insight into  Theorems \ref{mark}--\ref{main} and serve to establish ``trivial" bounds for $\lambda_z(\B)$.
The problem of finding points of $\mathcal{M}_z$ is  equivalent to the problem of  fitting geodesics in $\H$ between discs of varying  radii around the images 
under $\G$ of $z$. 

Consider the case of the original Markov spectrum $\mathcal{M}_\infty$.
A Ford circle is the horocycle around the reduced rational number $p/q$ with radius $ \frac{1}{2q^2}$.
The set of all Ford circles form a packing of the tessellation $\G \langle i, \rho\rangle$.
See the left hand side of Figure \ref{ford}.
It is obvious that every geodesic $S$ must intersect infinitely many Ford circles.
 This gives that $\lambda_\infty(\B) \geq 2,$ or $\mathcal{M}_\infty \subset [2,\infty).$
 Ford \cite{Fo}  proved that if we reduce the radii  of the Ford circles to any \[r\geq r_0=\tfrac{1}{\sqrt{5}q^2}\]  it still forces intersection
 but if $r<r_0$ there are geodesics that intersect no circle.
  See the right hand side of Figure \ref{ford}.

\begin{figure}[ht]
\includegraphics[width=2.in]{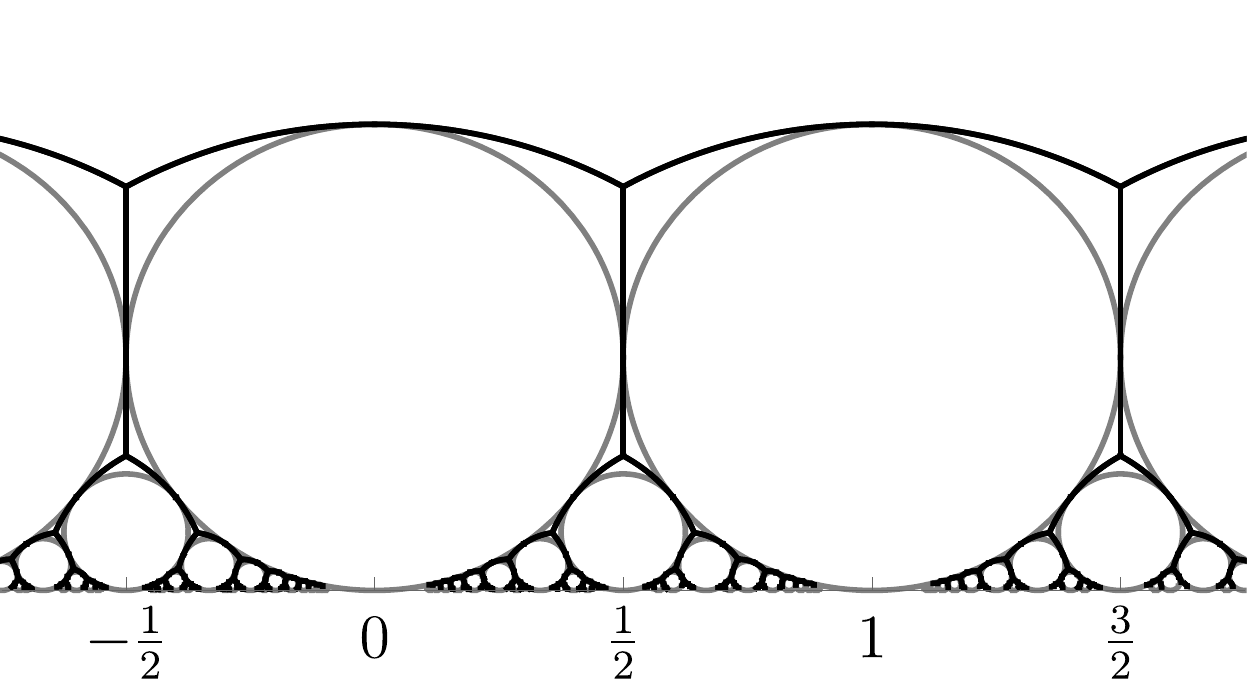} \qquad\qquad
\includegraphics[width=2.in]{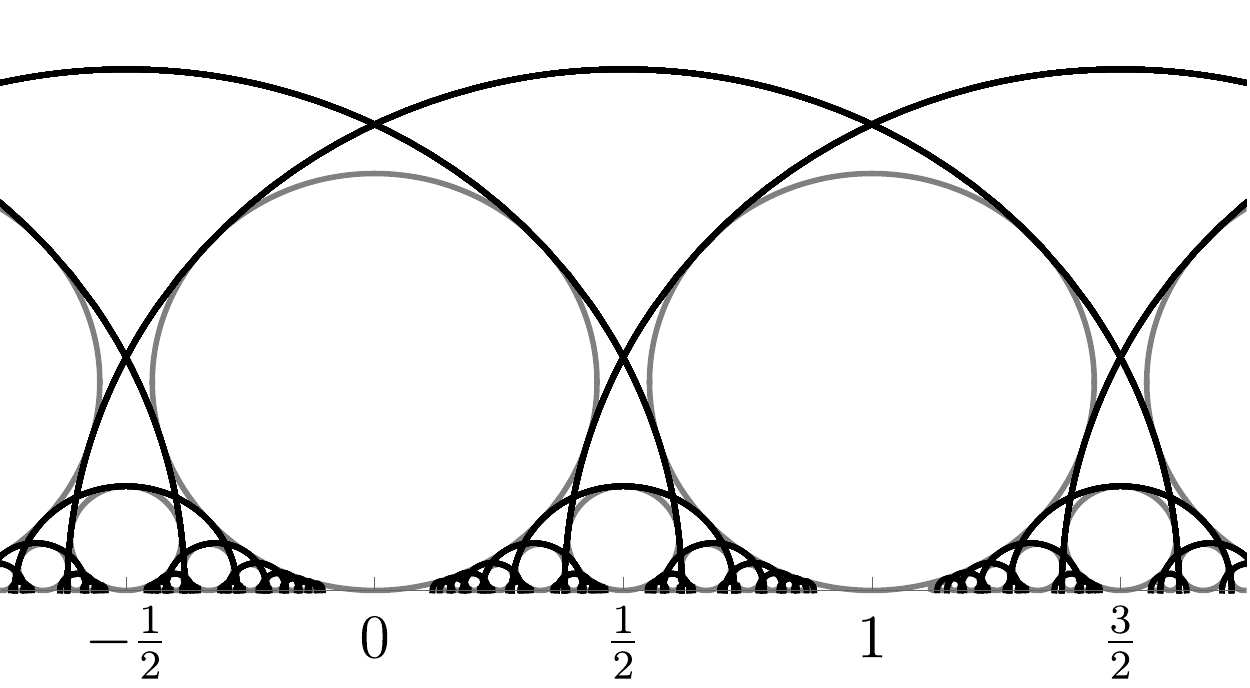}
         \caption{Ford circles}
    \label{ford}
\end{figure}

\subsubsection*{Proof of first statement of Theorem \ref{rhothm}}
The first statement of Theorem \ref{rhothm} may be proven this way.
To show that $\lambda_\rho(\B)\geq \sqrt{3}$ first observe that the hyperbolic
circles of radius $\half\log{3}$ around the points $\G\rho$ form tangent sequences that approach a dense subset 
of $\R$. It is straightforward to show that they are tangent to the Farey triangulation $\G \langle 0,\infty\rangle $. 
  Any geodesic  must  intersect these circles 
if their radius is made any larger since its endpoints will be separated by a sequence of circles.
See the left hand side of Figure \ref{r}.
\begin{figure}[ht]
\includegraphics[width=2.in]{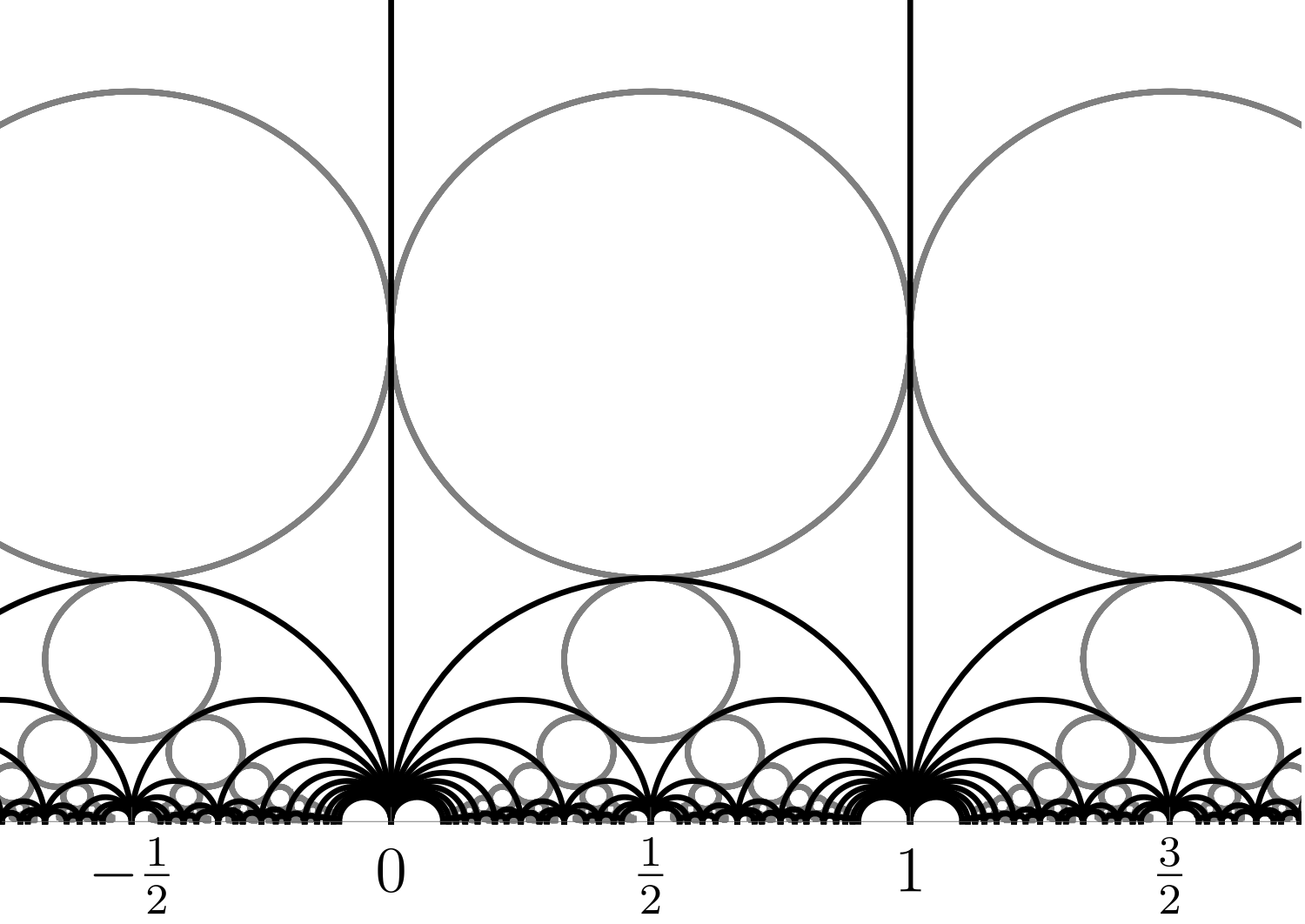} \qquad\qquad
\includegraphics[width=2.in]{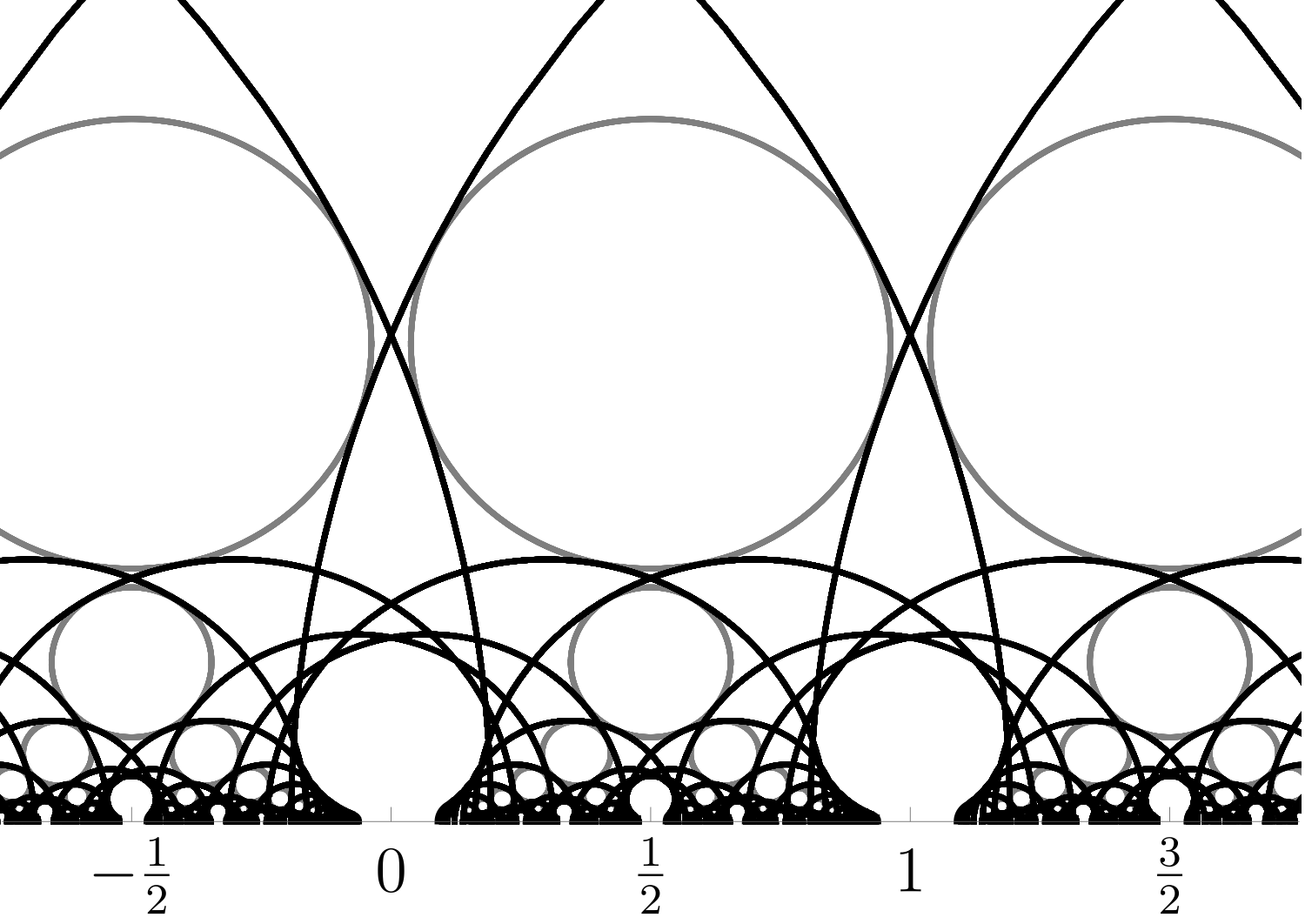}
         \caption{Illustrating Theorem \ref{rhothm} by  disks around images of $\rho$}
    \label{r}
\end{figure}

As we will prove below and is illustrated in the right hand side of Figure \ref{r}, if the radii are reduced by any positive amount there are infinitely many inequivalent geodesics 
that intersect no circle.  

\begin{figure}[ht]
  \includegraphics[width=2.in]{{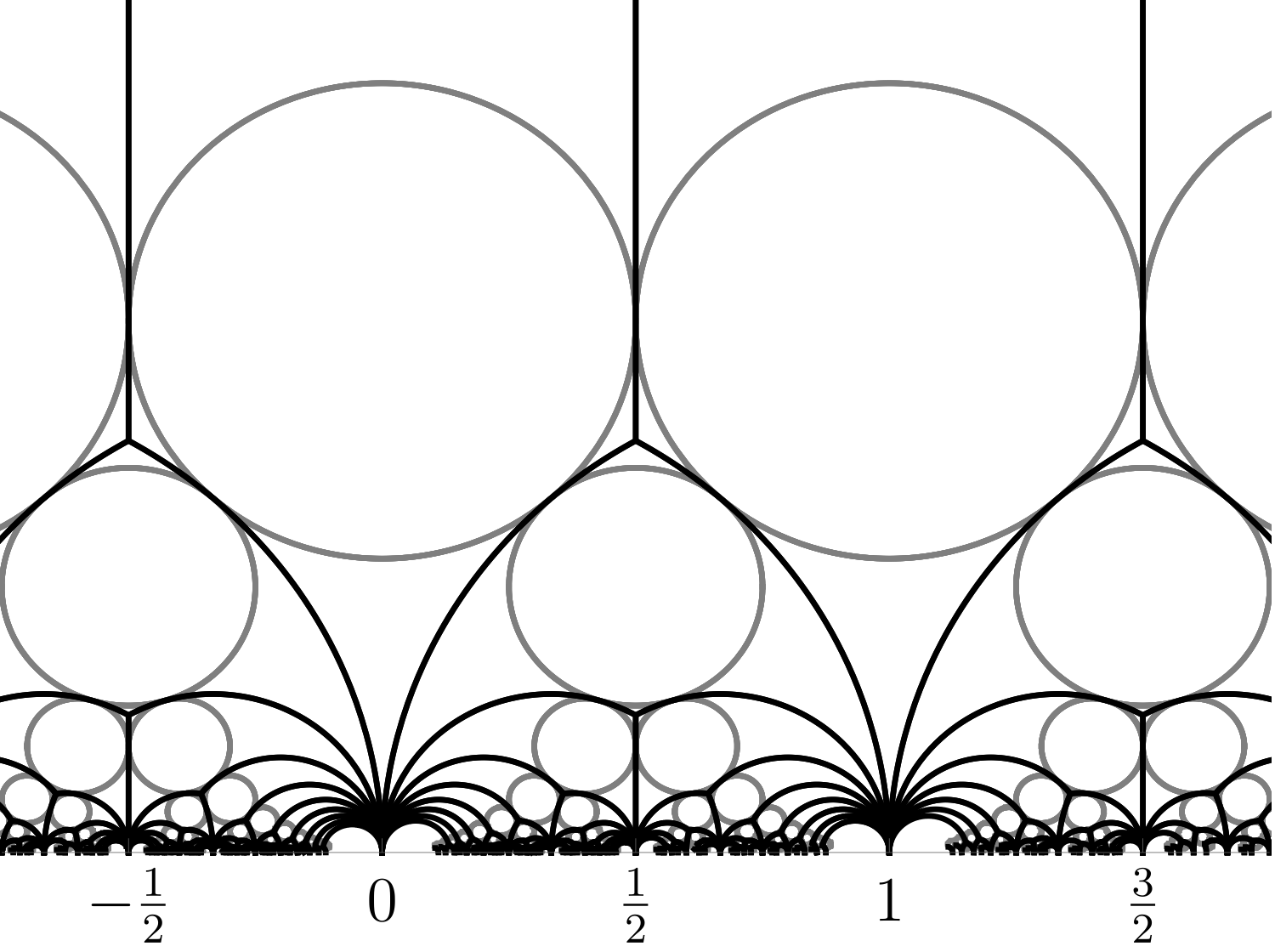}}   \qquad\qquad  
  \includegraphics[width=2.in]{{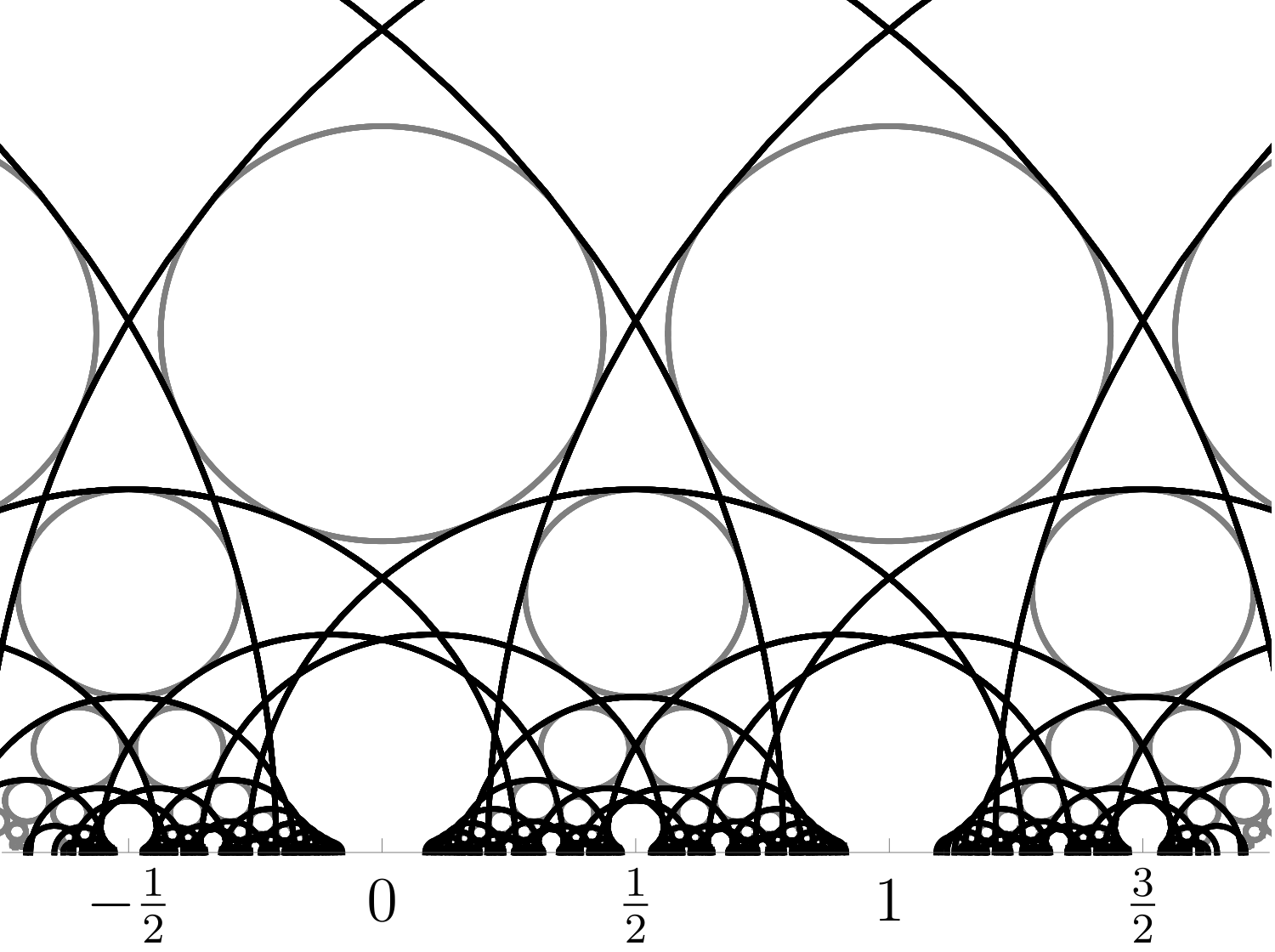}} 
    \caption{Disks around images of $i$ }
    \label{i}
\end{figure} 

Turning to Theorem \ref{main}, we can pack $\G \langle i,\infty \rangle $ by geodesic circles centered at the points $\G i$ of radius $\log(\frac{1+\sqrt{5}}{2})=0.481212\dots.$
This implies that $\lambda_i(\B) \geq 2$, as is illustrated in  the left hand side of Figure \ref{i}. This is weaker than the consequence of Theorem \ref{main} that $\lambda_i(\B) \geq 2.29129\ldots. $
as the right hand side of Figure \ref{i} illustrates.

This point of view sheds light on why the Markov-type result for the distance problem  is easier for $z=\rho.$ The corresponding tessellation 
in this case comprises complete geodesics, while in the other two cases  only geodesic segments.

For example, we easily get the first statement of the following result using the tessellation $\G\langle \frac{1}{2}, \infty \rangle=\G \langle i,\rho \rangle \cup\G \langle \rho,\infty\rangle$,
whose  associated billiard is $\mathcal{C}_{\frac{1}{2}}$. Figure \ref{tess0} illustrates the packing of this tessellation by disks around images of $\sqrt{-2}$ of radius $\frac{\log{2}}{2}$.

\begin{theorem}\label{2ithm}
 The smallest value in  $ \mathcal{M}_{2i}$ is  $\sqrt{8}$, which is attained by $\mathcal{C}_{\frac{1}{2}}.$  
 The value $\sqrt{8}$ is a limit point  of $ \mathcal{M}_{2i}$.
  \end{theorem} 

\begin{figure}[ht]
    \includegraphics[width=2.in]{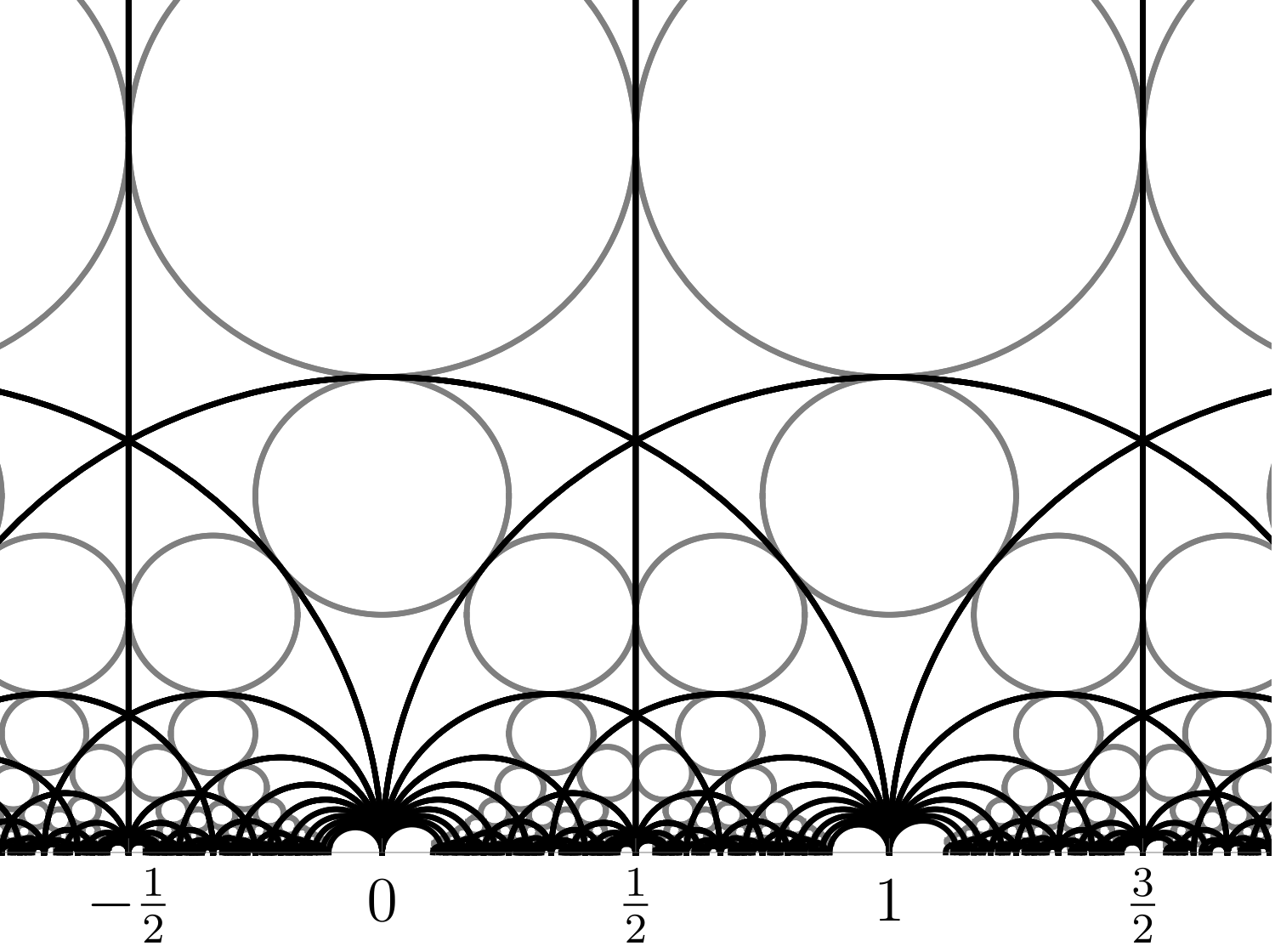}  \caption{Disks around images of $\sqrt{-2}$}
    \label{tess0}
\end{figure}

\section{Binary quadratic forms and billiards}\label{bqf}

To go beyond this basic geometric method we need a usable formula for the distance between a billiard and a point.  Binary quadratic forms provide the key.
In this section we establish their relation to modular billiards of the various kinds. 

For  $a,b,c\in \R$ with $d=\mathrm{disc}(Q)=b^2-4ac \neq 0$ let
\begin{equation}\label{Q}
Q(x,y)=ax^2+bxy+cy^2,
\end{equation}
which is  a non-singular real binary quadratic form.
Sometimes we will write  $Q=(a,b,c).$
Now $M=\pm \pmatrix{a'}{b'}{c'}{d'}\in \G= \mathrm{PGL}(2,\Z)$ acts on $Q$ by
\begin{equation}\label{equiv}
(Q|M)(x,y)\stackrel{\text{def}}{=}(\det M) Q(a'x+b'y,c'x+d'y).
\end{equation}
Clearly $Q|(M_1M_2)=(Q|M_1)|M_2$ for $M_1,M_2 \in \G$.
We say that two such forms 
$Q_1$ and $Q_2$ are {\it  equivalent }
if there is an $M\in \G$ such that
\[
(Q_1|M)(x,y)=Q_2(x,y).
\]
If $M\in \mathrm{PSL}(2,\Z)$ then we say that $Q_1$ and $Q_2$ are {\it  properly equivalent.}
 The class of forms that are equivalent to $Q$, but not necessarily properly equivalent to $Q$,  will be denoted by $[Q]$.
The discriminant $\mathrm{disc}(Q)$ is an invariant of $[Q]$.

If $Q=(0,b,c)$ with $b>0$  let $\a_Q=-\frac{c}{b}$ and $\b_Q=\infty$, while if $b<0$ with  let $\b_Q=-\frac{c}{b}$ and $\a_Q=\infty$.
Otherwise the roots of $Q(z,1)=0$   are given by
\begin{equation}\label{first}
 \a_Q=\frac{-b+\sqrt{d}}{2a} \;\;\;\mathrm{and}\;\;\;\b_Q=\frac{-b-\sqrt{d}}{2a}.
\end{equation}
In all cases  $\a_Q$ will be called the {\it first} root and $\b_Q$ the {\it second} root
of $Q$.
One checks that $d$, $\a_Q$ and $\b_Q$ uniquely determine $Q.$
Furthermore, using the generators $A,B,C$ from \eqref{ABC}, it follows
that for each $j=1,2$ and for any $M=\pm \pmatrix{a'}{b'}{c'}{d'}\in \G$
\begin{equation}\label{ff}
\a_{Q|M}=M^{-1}(\a_Q)\;\;\;\mathrm{and}\;\;\;\b_{Q|M}=M^{-1}(\b_Q),
\end{equation}
with $M(z)$ given
in the definition around \eqref{act} extended to all of $\C$.
It is important that the action of $\G$ on quadratic forms defined in \eqref{equiv} preserves the first and second roots.

The proofs of the following two results are straightforward.
\begin{proposition}\label{firc0}

(i) The map $Q \mapsto \a_Q$ determines a bijection  between classes $[Q]$ with a fixed negative discriminant that are represented by positive definite $Q$ and points of $ \T$.

\smallskip\noindent
(ii) The map  $Q \mapsto \langle \a_Q,\b_Q\rangle $ determines a bijection $[Q] \leftrightarrow \B$ between the classes $[Q]$ with a fixed positive discriminant and the set of all  modular billiards, where the  associated billiard  $\B$ is induced by $\langle \a_Q,\b_Q\rangle $.
\end{proposition}
Say $Q$ {\it represents zero} if $Q(x,y)=0$ for some $x,y\in \Z$ not both zero, that $Q$ is {\it reciprocal} if $Q$ is equivalent to $-Q$
and that $Q=(a,b,c)$ is {\it primitive integral} if $a,b,c\in \Z$ with $\gcd (a,b,c)=1.$
We have the following characterization of improper, non-orientable and periodic billiards in terms of quadratic forms.

\begin{proposition}\label{firc}
 Under the bijection $[Q] \leftrightarrow \B$ of (ii) of Proposition~\ref{firc0},  for any $Q\in [Q]$ 
\begin{enumerate}[label=(\alph*)]
\item
$Q$ represents zero if and only if $\B$ is improper,
\item
$Q$ is reciprocal if and only if $\B$ is non-orientable,
\item For some non-zero real $\kappa$ the form $\kappa Q$ does not represent zero  and is primitive integral
if and only if $\B$ is periodic.    \end{enumerate}
\end{proposition}
 
 The study of  periodic billiards is the same as that of  (primitive) integral binary quadratic forms.
 It follows from Proposition~\ref{firc} (c) that each periodic modular billiard $\B$ may be assigned a unique positive integer
given by  \[d=\mathrm{disc}(\B)\stackrel{\text{def}}{=}\kappa^2\mathrm{disc}(Q)\] for any $Q\in [Q]$.
There are only finitely many periodic billiards with a given discriminant $d$
and each has the same length.

A formula for the length of $\B$  is determined by finding that  solution $(t,u)$ with $t,u \in \Z^+$ of
$t^2-du^2=\pm 4$
for which 
$\varepsilon= \tfrac{1}{2}(t+u\sqrt{d})$ is minimal. Then the length is $2\log \varepsilon.$
If $t^2-du^2=-4$  has a solution  then proper equivalence is the same as equivalence.
Otherwise each ordinary class consists of two proper classes.

Classically one says that a primitive integral $Q$ is improperly equivalent to itself when
 $Q|M=-Q$ where $M\in \G$ with $\det M=-1$ since then the change of
 variables $(x,y)\mapsto (x,y)M^t$ preserves $Q$.
It follows from Theorem 90 of \cite{Di} that this holds if and only if  some  $(a,b,c)\in [Q]$  has $a|b$.
 These forms are called {\it ambiguous} and one may say that the billiard $\B$ associated to $[Q]$ is ambiguous.
By  Proposition~\ref{firc} (b) we have that an ambiguous $\B$ is non-orientable. If $t^2-du^2=-4$  has a solution  then every non-orientable billiard is ambiguous. 
Otherwise
it is possible for a billiard to be non-orientable without being ambiguous. This happens when an associated $Q$ is properly equivalent to $-Q$.
Markov's billiards are examples.
If $d$ is fundamental, the number of non-orientable billiards is $2^{\omega(d)-1}$, where $\omega(d)$ is the number of distinct prime factors dividing $d$.

\section{Quaternary quadratic forms}\label{qqf}
A consequence of Proposition~\ref{firc},  one that is crucial for us,  is a  formula for \[\sinh \delta_z(\B)=(\lambda_z(\B))^{-1}\] from \eqref{mar2}. This involves finding the minimum of a certain indefinite quaternary quadratic form. 
For a fixed $z\in \T$ let $Q'(x,y)=a'x^2+b'xy+c'y^2$ with $a'>0$ and $d'=b'^2-4a'c'<0$ represent $z$. Also,  let $Q(x,y)=ax^2+bxy+cy^2$ with $d=b^2-4ac>0$ represent the modular billiard $\B$.

\begin{proposition}\label{quad}
Notation as above,
\begin{equation*}
\sinh \delta_z(\B)=  (d|d'|)^{-\frac{1}{2}}\inf_{[Q]}|2c'a+2a'c-b'b|.
\end{equation*}
\end{proposition}
\begin{proof}
A standard exercise in hyperbolic geometry shows that the hyperbolic distance $\delta(z,S)$ from $z$ to the geodesic $S=\langle \alpha_Q,\beta_Q\rangle$ satisfies
\begin{equation}\label{quadf}
\sinh \delta(z,S)=(d|d'|)^{-\frac{1}{2}}|2c'a+2a'c-b'b|.
\end{equation}
See e.g.~\cite[p.~162]{Be}. The result follows.
\end{proof}

We are now able to justify the second statements of Theorems \ref{rhothm} and \ref{2ithm}.  For fixed $ \ell \in \Z^+$ let $\mathcal{A}_\ell$ be the billiard associated to the quadratic form
\[
Q(x,y)=x^2-\ell xy-y^2.
\]
The case when $\ell=5$ is illustrated in Figure \ref{fig1}, which is typical in that of those geodesics in the orbit  intersecting $\T$, the one corresponding to $Q$ is the one that  gives the closest approach to $\rho.$
By Proposition~\ref{quad} we have
\begin{equation*}
\sinh \delta_\rho(\mathcal{A}_\ell)=  \frac{\ell}{\sqrt{3\ell^2+12}}.
\end{equation*}
Therefore $\lambda_\rho(\mathcal{A}_\ell)=\frac{\sqrt{3\ell^2+12}}{\ell},$ which decreases to the limit $\sqrt{3}$ as $\ell \rightarrow \infty$.
This  completes the proof of Theorem \ref{rhothm}.
 The second statement in Theorem \ref{2ithm} follows in like manner.

Note that we may rewrite \eqref{quadf} using \eqref{mar2} as
\begin{equation}\label{quat}
\lambda_z(\B)^{-1}= (d|d'|)^{-\frac{1}{2}}\inf_{x_1x_4-x_2x_3=\pm1} |Q''(x_1,x_2,x_3,x_4)|,
\end{equation}
where  for a {\it fixed} choice of $Q(x,y)=ax^2+bxy+cy^2$ representing  $\B$ we have
\begin{align*}\label{Q''}
Q''(x_1,x_2,x_3,x_4)=&2a'cx_1^2+2a'ax_2^2+2c'cx_3^2+2c'ax_4^2\\\nonumber &+2b'ax_2x_4+2b'cx_1x_3-2c'bx_3x_4-2a'bx_1x_2-b'bx_1x_4-b'bx_2x_3.
\end{align*}
Here $Q''$  is an indefinite quaternary quadratic form of signature $(2,2)$.
Observe  that we must minimize $Q''(x_1,x_2,x_3,x_4)$ subject to 
\begin{equation}\label{sc}
x_1x_4-x_2x_3=\pm1
\end{equation}
whereas it is more usual  to only require that $(x_1,x_2,x_3,x_4)\neq (0,0,0,0).$

The  formula for  $\lambda_\infty(\B)$ from the Markov spectrum corresponding to \eqref{quat} is simply
\begin{equation}\label{ma}
\lambda_\infty(\B)^{-1}= d^{-\frac{1}{2}}\inf_{(x_1,x_2)\neq (0,0)} |Q(x_1,x_2)|.
\end{equation}
In this sense the problem of finding $\lambda_i(\B)$ is more difficult than that of finding $\lambda_\infty(\B)$.
The study of the  minima of certain indefinite quaternary forms subject to \eqref{sc} goes back at least to a 1913
paper of Schur \cite{Sc}, which was an inspiration for this paper and deserves to be better known.

At this point we may obtain a good lower bound for $\lambda_i(\B)$ 
 when $\B$ is improper.
It is easy to check that an improper billiard is determined by some $\langle\a,\infty\rangle$, where $0\leq \a\leq \frac{1}{2},$
or equivalently by the form $Q=y(x-\a y)$ for this $\a.$
\begin{proposition}\label{imp}
For $\B$ an improper billiard we have that
\begin{equation*}
\lambda_i(\B)\geq  3
\end{equation*}
and this   is attained by the billiard determined by 
$\langle\frac{1}{3},\infty \rangle$.
\end{proposition}\
\begin{proof}
Let $Q(x,y)=y(x-\a y)$. Then $\lambda_i(\B)^{-1}\leq  \min( \a, 1-2\a)$,
which is found by applying Proposition~\ref{quad} in the form \eqref{quat}  to 
\[
\half Q''(x_1,x_2,x_3,x_4)=-\a( x_1^2+x_3^2)-x_3x_4-x_1x_2
\]
and taking $(x_1,x_2,x_3,x_4)=(1,0,0,1), (1,-1,1,0)$. Thus $\a=1/3$ gives the minimum.
\end{proof}

\section{Proper billiards and reduced forms}\label{proper}

We say that a form $Q=(a,b,c)$ with discriminant $d>0$ that does not represent zero is {\it reduced}  if 
\[
-1<\b_Q<0 \;\;\;\mathrm{and}\;\;\;\a_Q>1,
\]
where $\a_Q$ and $\b_Q$ were defined in \eqref{first}.
A classical argument given in the proof of Theorem 76 in \cite{Di} may be adapted to prove
that for any proper billiard $\B$ the corresponding class $[Q]$ (as in Proposition~\ref{firc}) contains such reduced forms. 

Given a  form $Q=(a,b,c)$, let
\begin{equation}\label{qstar}
Q^*(x,y)=-Q(-y,x).
\end{equation}
If $Q$ is reduced then $Q^*$ is 
a reduced form that is properly equivalent to $-Q.$

Note  that the geodesics associated to the reduced forms do not necessarily account for all of the geodesic segments comprising a modular billiard. This fact is illustrated in Figure \ref{fig1}, where the single geodesic associated to a reduced form is shown in black. In the proof of Markov's Theorem \ref{mark}
 the maximal height of a billiard $\B$  will be approached by the heights of geodesics associated to reduced forms. Thus by \eqref{ma}  it follows that we have the simple formula
\begin{equation}\label{a}
\lambda_\infty(\B)^{-1}=  d^{-\frac{1}{2}}\inf_{Q\in [Q] \;\;\mathrm{reduced}} |a|.
\end{equation}

To obtain an analogous formula for $\lambda_i(\B)^{-1}$
we must consider some transforms of reduced forms.
This motivates the following definition.  For any for $Q=(a,b,c)$ with discriminant $d>0$ define
 \begin{equation}\label{nu}
 \nu(Q)=d^{-\frac{1}{2}}\min \big( |a+c|,|2a+b+c|,|2c+b+a|\big).
 \end{equation}

\begin{proposition}\label{reduc}
For a proper $\mathcal B$ we have
\begin{equation}
	\lambda_i(\mathcal B)^{-1} = \inf_{Q\in [Q] \text{ reduced}} \nu(Q),
\end{equation}
where the class $[Q]$ corresponds to $\mathcal B$.
\end{proposition}

\begin{proof}
To prove this we will show that for any geodesic in $\H$ that intersects $\mathcal T$, we can find a
form $Q = (a, b, c)$ such that either $Q$ or $-Q$ is reduced and that the geodesic corresponding
to $Q(x, y)$, $Q(x + y, y)$, or $-Q(x, x - y)$ is as close or closer to $z = i$. 
The result then follows by Proposition~\ref{quad} applied with $Q' = (1,0,1)$ since the terms in \eqref{nu} correspond exactly to these three cases.

Note that we may restrict our attention to geodesics $S$ that either (i) cross both vertical
sides of the boundary of $\mathcal T$ or (ii) cross the right vertical side and the circular arc of this
boundary. 
This is because the reflection across the $y$-axis of a geodesic that crosses the left
hand vertical side and the circular arc will cross both vertical sides and will also have the
same distance from $z = i$.

In case (i) we may assume that the apex of the geodesic $S$ lies on or to the right of the $y$-axis; if not, the reflection across the $y$-axis of $S$ will have that property and be the same distance from $i$.
Let $\til \alpha,\til \beta$ be the roots of a form $\til Q$ associated to $S$.
Then we have
\begin{equation*}
	\til \alpha > \tfrac 12, \qquad \til \beta < 0, \qquad \til \alpha - \til \beta \geq 2, \quad \text{ and } \quad \til \alpha + \til \beta \geq 0.
\end{equation*}
Thus there is a unique integer $n\geq 0$ such that the form $Q(x,y) = \til Q(x-ny,y)$, obtained by shifting $S$ to the right $n$ units, is reduced.
If $n=0$ or $1$ then we are done because either $\til Q$ is reduced already or $\til Q(x,y) = Q(x+y,y)$.
Suppose that $n\geq 2$.
Then $\til \beta < -2$ and $\til \alpha >2$ which implies that the roots $\alpha,\beta$ of $Q$ satisfy
\begin{equation*}
	-1 < \beta < 0 \qquad \text{ and } \qquad \alpha > 4.
\end{equation*}
If $\langle \alpha,\beta \rangle$ intersects the $y$-axis above $i$, then a simple geometric argument shows that $\langle \alpha,\beta \rangle$ is closer than $S$ to the point $z=i$.
If not, then the geodesic $\langle \alpha-1, \beta-1 \rangle$ associated to $Q(x+y,y)$ crosses the $y$-axis above $i$ (since $(\alpha-1)(\beta-1)\leq -2$); hence either $\langle \alpha-1, \beta-1 \rangle$ or $\langle \alpha, \beta \rangle$ is closer than $S$ to the point $z=i$.

In case (ii) we may assume that either $Q = \til Q$ is reduced or that
\begin{equation*}
	0<\til \beta < \tfrac 12 \qquad \text{ and  }\qquad \til\alpha > 1,
\end{equation*}
in which case $Q(x,y) = -\til Q(x,x-y)$ is reduced.
Since this is equivalent to the identity 
\begin{equation*}
	\til Q(x,y) = -Q(x,x-y),
\end{equation*}
we are done.
\end{proof}

\section{Billiards and Sequences}\label{seq}

Simple continued fractions are crucial  in Markov's proof of  Theorem \ref{mark} and in our proof of Theorem \ref{main}. We denote one by
\[
[k_1,k_2,k_3,k_4,\dots]= k_1+\cfrac{1}{k_2+\cfrac{1}{k_3+\cfrac{1}{k_4+\dots}}},
\]
 whose finite version ending in $\frac{1}{k_n}$ is written
$[k_1,k_2,\dots,k_n].$
Here we will present the beginnings of Markov's method in a form that we  adapt in the next section to prove Theorem \ref{main}.
The method relates chains of reduced quadratic forms to doubly-infinite sequences of positive integers.
We give a somewhat novel treatment of this correspondence based on equivalence rather than proper equivalence.

For a reduced  $Q$  define the doubly infinite sequence $K_Q$ by expanding
\begin{equation}\label{a,b}
\a_Q=[k_1,k_2,k_3,\dots]\;\;\;\mathrm{and}\;\;\; -\b_Q=[0,k_0,k_{-1},k_{-2},\dots]
\end{equation}
into simple continued fractions
and setting
 \begin{equation}\label{eq:map}K_Q=(\dots, k_{-1},k_0,k_1,k_2,\dots).
 \end{equation}
We shall refer to $k_1$ as the first entry of $K$.

Say two doubly infinite sequences of positive integers 
$K=(k_n)$ and $L=(\ell_{n})$ are {\it equivalent} if there is a $j \in \Z$  such that $k_{ n}=\ell_{n+j}$ for all $n\in \Z$
and {\it properly equivalent} if there is a  $j \in 2\Z$  such that $k_{ n}=\ell_{n+j}$ for all $n\in \Z$.
 If $K=(k_n)$ define 
the {\it reversal} of $K$ to be  $K^*=(k_{1-n})$.

\begin{proposition}\label{prop:map}
The map $Q\mapsto K_Q$ determines a bijection between classes of  forms  with a fixed positive discriminant that do not represent zero and equivalence classes  of sequences of positive integers. It also determines a bijection between proper equivalence classes of  forms  with a fixed positive discriminant that do not represent zero and proper equivalence classes of sequences of positive integers.
Furthermore $K^*_Q=K_{Q^*}. $
\end{proposition}

\begin{proof}
For each  $n \in \Z$ let
\begin{equation}\label{rs}
r_n=[k_n,k_{n+1},\dots]\;\;\; \mathrm{and}\;\;\;s_n=[0,k_{n-1},k_{n-2},\dots].
\end{equation}
Thus $r_1=\a_Q$ and $s_1=-\b_Q$ and also
\begin{equation}\label{recu}
r_{n-1}=\frac{1}{r_{n}}+k_{n-1} \;\;\;\mathrm{and}\;\;\; s_{n-1}=\frac{1}{s_{n}}-k_{n-1}.
\end{equation}
Define $a_n,b_n > 0 $ for each $n \in \Z$ by
 \begin{equation}\label{r,s}
\frac{a_{n}}{\sqrt{d}}=\frac{1}{r_{n-1}+s_{n-1}}\;\;\;\;\;\mathrm{and}\;\;\; \;\;\;\frac{b_{n}}{\sqrt{d}}=\frac{r_n-s_n}{r_n+s_n}.
\end{equation}
It can be seen that  using \eqref{recu} that $d=\mathrm{disc}(Q)=b_n^2+4a_{n}a_{n+1}$ for all $n \in \Z.$
Let
\begin{equation}\label{chain}
Q_n=\big(a_{n+1},-b_n,-a_{n}\big).
\end{equation}
A calculation shows that
\[\a_{Q_n}=r_n\;\;\;\mathrm{and}\;\;\; -\b_{Q_n}=s_n.\]
It follows that each $Q_n$ is reduced and equivalent to $Q$.
We claim that every such form occurs as a $Q_n$. Further, each $Q_{2n+1}$ is properly equivalent 
to $Q$ and every reduced  form that is properly equivalent to $Q$  
is one of the $Q_{2n+1}$. Again, these statements follow from variations on the arguments given in  Chaper VII of \cite{Di}.
That the claimed bijections are well-defined and injective follows.
 Clearly every sequence $K$ arises from some reduced form so the maps are also surjective.

Turning to the last statement, recall that $Q^*$ was defined in \eqref{qstar} and
 observe that
 \[\b_{Q^*}=-\tfrac{1}{\a_Q}\qquad\text{and}\qquad \a_{Q^*}=-\tfrac{1}{\b_Q}.\] Let $r_n^*$ and $s_n^*$ correspond to $Q^*$ as in \eqref{rs}.  By \eqref{a,b} we have 
\[
r^*_1=-\tfrac{1}{\b_Q}=[k_0,k_{-1},k_{-2},k_{-3},\dots]\;\;\;\mathrm{and}\;\;\;s^*_1=\tfrac{1}{\a_Q}=[0,k_1,k_2,k_3\dots],
\]
giving the result.

This completes the proof of Proposition~\ref{prop:map}.
\end{proof}

Say that a sequence $K=(k_n)$ is {\it periodic}  if there is an $N\in \Z^+$ so that $k_{n+N}=k_n$ for all $n\in \Z$
and  {\it palindromic } if $K^*$ is equivalent to $K$.
Combining Propositions \ref{firc} and \ref{prop:map} we derive the following correspondence.
\begin{theorem}\label{map2}
There is  a bijection between proper modular billiards and equivalence classes of doubly infinite sequences of positive integers. Under this correspondence a billiard is periodic precisely when 
the sequence is periodic and non-orientable precisely when the corresponding sequence is palindromic.
\end{theorem}

From the first formula of \eqref{r,s} and \eqref{a} we have
\begin{equation}\label{b}
\lambda_\infty(\B)=\sup_{n\in \Z}(r_n+s_n).
\end{equation}
This is the starting point of the proof of Markov's Theorem \ref{mark}.
The main difficulty is in understanding which $K$ cannot have any small values of $r_n+s_n$.
The first observation is that if any $k_m>2$ we must have  $\lambda_\infty(\B)>3$.
The complete result requires an ingenious  analysis of continued fractions
all of whose partial quotients are either 1 or 2. A  treatment of Markov's method and a  proof of Theorem \ref{mark} based on it can be found in Dickson's book~\cite{Di2}. Other useful references are \cite{Ai}, \cite{Ber}, \cite{Bo} and \cite{Ca}.

\section{Sequences and the spectrum}\label{seqspec}

To return to the proof of Theorem \ref{main}, 
recall from \eqref{genmar} that  $\mathcal{M}_i$  is defined in terms of $\lambda_i(\B)$, which was given in \eqref{mar2}.
We now  find a formula for $\lambda_i(\B)$ that is analogous to \eqref{b} when $\B$ is a proper modular billiard.
For our problem we are led to  estimate certain quantities  involving  pairs of  successive 
values of $r_n$ and $s_n$ from \eqref{rs}, rather than simply $r_n+s_n.$

 Let $K$ be a doubly-infinite sequence of positive integers.
The quantities we need are the following:
\begin{align}
\mu'_n(K)&=\left| \frac{1}{r_n+s_n}-\frac{1}{r_{n-1}+s_{n-1}}\right|\label{mup}\\
\mu''_n(K)&=\left| \frac{2-r_n+s_n}{r_{n}+s_{n}}-\frac{1}{r_{n-1}+s_{n-1}}\right|\label{mupp}\\
\mu'''_n(K)&=\left| \frac{2}{r_{n-1}+s_{n-1}}+\frac{r_n-s_n-1}{r_{n}+s_{n}}\right|.\label{muppp}
\end{align}
Also set

\begin{equation}\label{mu}
\mu'(K)= \inf_{n \in \Z}\mu'_n(K),\;\;
\mu''(K)=\inf_{n \in \Z}\mu''_n(K),\;\;
\mu'''(K)=\inf_{n \in \Z}\mu'''_n(K)
\end{equation}
and
 \begin{equation}\label{max}
 \mu(K)= \min \big(\mu'(K),\mu''(K),\mu'''(K),\mu'(K^*),\mu''(K^*),\mu'''(K^*)\big).
 \end{equation}
 
\begin{proposition}\label{finl}
Let  $K$ correspond to a proper $\B$. Then we have
\[\lambda_i(\B)^{-1}= \mu(K).\]
\end{proposition}
\begin{proof}
By  Proposition~\ref{prop:map} and \eqref{r,s} 
\[\mu'_n(K)=d^{-\frac{1}{2}}|a_{n+1}-a_{n}|,\;\;\mu''_n(K)=d^{-\frac{1}{2}}|2a_{n+1}-b_n-a_{n}|,\;\;\mu_n'''(K)=d^{-\frac{1}{2}}|2a_{n}+b_n-a_{n+1}|.\]
The result follows from Proposition~\ref{reduc} since every reduced form of the class $[Q]$ is found among the $Q_n.$
\end{proof}

In the following proposition we  show that  
\[
\{\half\sqrt{21},\;\;\tfrac{2}{3}\sqrt{14},\;\;\tfrac{1}{3} (3+\sqrt{21})\}\in\mathcal{M}_i,
\] 
and that each value is attained.
For $j=1,2,3$ and some fixed choice of the first  entry in each,  define doubly-infinite sequences
\begin{equation}\label{K}
K_1=(\overline{1,3}),\quad K_2=(\overline{1,2,1,6}),\quad K_3=(\overline{3,1},4,\overline{1,3}).
\end{equation}
Here an overlined subsequence adjacent to a parenthesis indicates that one must concatenate the subsequence infinitely many times in the direction of the parenthesis.

\begin{proposition}\label{bj}
Let $\sC_j$ be the modular billiard associated to $K_j$ for $j=1,2,3$. Then
\begin{equation}\label{cs}
	\lambda_i(\sC_1) = \tfrac 12 \sqrt{21}, \qquad \lambda_i(\sC_2) = \tfrac 23 \sqrt{14}, \qquad \lambda_i(\sC_3) = \tfrac 13 (3+\sqrt{21}).
\end{equation}
Each value $\lambda_i(\sC_j)$ is attained. All three billiards are proper and non-orientable; $\sC_1, \sC_2$ are periodic, while $\sC_3$ is not periodic.
\end{proposition}
\begin{proof}
That \eqref{cs} holds and that each $\lambda_i(\sC_j)$ is attained  is a straightforward application of Proposition~\ref{finl}.
Clearly $K_1,K_2$ are periodic and palindromic and $K_3$ is not periodic but is palindromic. Thus the final statement follows from Theorem \ref{map2}.
\end{proof}

 \section{Exceptional sequences }\label{ex}

In this and the next section we complete the
proof of  Theorem \ref{main}.
 By Proposition~\ref{imp} we may  assume that the billiard $\B$ is proper.
The main result of this section is the following proposition, which (together with Proposition~\ref{finl}) shows that the only proper billiards that stay farther away from $z=i$  than $\mathcal{C}_3$  are $\mathcal{C}_2$ and $\mathcal{C}_1$.
The method is completely elementary and amounts to finding inequalities  determined by continued fractions.
 
\begin{proposition}\label{mp}
Unless $K$ is equivalent to $K_j$ for $j=1,2,3$
 we have that \[\mu(K) < \tfrac{1}{4} \left(\sqrt{21}-3\right)=0.395644\dots.\]
\end{proposition}

We will say that any $K$ with $\mu(K) \geq 0.395644\dots$ is {\it exceptional.}
The proof of Proposition~\ref{mp} consists of  a series of results that successively eliminate
 configurations of subsequences in a $K$ that force it to not  be exceptional.
 Since $\mu(K)=\mu(K^*)$,  it is clearly permissible to only prove it for either $K$ or $K^*$.  Hence we will often  only provide estimates for 
  one of them and might not mention when reversals 
 must also be considered in order to cover all cases.  
 
\begin{proposition}\label{K2}
An  exceptional $K$ must have the form
\[
K=(\dots, 1, m_1,1,m_2,1,\dots)
\]
where $m_j \geq 2$.
If any $m_j=2$ then 
$K$ is equivalent to $K_2=(\overline{1,2,1,6}).$
\end{proposition}
Proposition~\ref{K2}  will be proven in the four lemmas that follow.
\begin{lemma}\label{l0}
Given any $K$, if $K$ contains any subsequence of the form 
\begin{equation}\label{list}
(1,1),(2,2),(2,3),(2,4),(2,5),(2,6),(m,m')
\end{equation}
 for $3\leq m\leq m' $  
then $\mu'(K) \leq \frac{1}{3}$ and therefore $K$ is not exceptional.
\end{lemma}
\begin{proof}
If $K$ contains the subsequence $(m,m')$ 
for  any $m,m'\in \Z^+$ there is an $n \in \Z$ so that
\[
r_{n-1}=[m,m',\dots],\;r_n=[m',\dots],\;s_{n-1}=[0,\dots],\; s_{n}=[0,m,\dots].
\]
Thus $[m,m',1]<r_{n-1}<[m,m'],\; 0<s_{n-1}<1,\; [m']<r_{n}<[m',1],\;[0,m,1]<s_n<[0,m]$
and so it follows that
\[
\frac{m}{mm'+m+1}<\frac{1}{r_n+s_n}<\frac{m+1}{mm'+m'+1}
\]
and
\[
\frac{m'}{mm'+m'+1}<\frac{1}{r_{n-1}+s_{n-1}}<\frac{m'+1}{mm'+m+1}.
\]
The result now follows from \eqref{mup} and \eqref{mu}  since for the pairs in \eqref{list}
it can be easily verified that
\[
\left|\frac{m-m'-1}{mm'+m+1}\right|,\left|\frac{m'-m-1}{mm'+m'+1}\right|\leq\frac{1}{3}. \qedhere
\]
\end{proof}
A useful consequence  of Lemma \ref{l0}  is that  $(1,1)$ cannot occur as a subsequence  in an exceptional $K$. We will in several places use this fact without  further mention.

\begin{lemma}\label{l1}
If $K$ contains  the sequence $(\ell',2,\ell)$ where $\ell\neq 1$ or $\ell'\neq 1$ then $K$ is not exceptional.
\end{lemma}
\begin{proof}
By Lemma \ref{l0} we may assume that either (i) $\ell,\ell' \geq 7$ or (ii) that $\ell'=1$ and $\ell\geq 7.$
In case (i)
there is an $n \in \Z$ so that 
\[r_{n-1}=[\ell',2,\ell,\dots],\;\;r_n=[2,\ell,\dots],\;\;s_{n-1}=[0,\dots],\;\;s_n=[0,\ell',\dots],\]
hence
\[
r_{n-1}>[7,2,7]=\tfrac{112}{15},\;\;0<s_{n-1}<1,\;\;2<r_n<[2,7]=\tfrac{15}{7},\;\;0<s_n<[0,7]=\tfrac{1}{7}.
\]
It follows that
\[
0<\frac{1}{r_{n-1}+s_{n-1}}<\tfrac{15}{112}\;\;\;\mathrm{and} \;\;\;-\tfrac{1}{16}<\frac{2-r_n+s_n}{r_n+s_n}< \tfrac{1}{14}
\]
so that by \eqref{mupp} we have that $\mu''_n(K)<\tfrac{11}{56}=0.196429\dots.$

In case (ii) either (a) $K$ contains $(m',1,m,1,2,\ell)$ where $m,m'\geq 2$ or
(b) $K$ contains $(m,2,1,2,\ell)$ where $m\geq 7$.

In case (a) 
there is an $n \in \Z$ so that 
\[r_{n-1}=[m,1,2,\ell,\dots],\;\;r_n=[1,2,\ell,\dots],\;\;s_{n-1}=[0,1,m',\dots],\;\;s_n=[0,m,1,m',\dots].\]
Thus
\[m+\tfrac{2}{3}<r_{n-1}<m+\tfrac{15}{22},\;\;\; \tfrac{2}{3}<s_{n-1}<1,\;\;\;\tfrac{22}{15}<r_{n}<\tfrac{3}{2},\;\;\; \tfrac{1}{m+1}<s_{n}<\tfrac{1}{m+\frac{2}{3}}.\]
Hence
\[
\frac{(m+3)(3m+2)}{3(m+1)(3m+4)}<\frac{2-r_n+s_n}{r_n+s_n}<\frac{(m+1)(24m+61)}{(3m+2)(22m+37)}
\]
and
\[
\frac{1}{m+\frac{37}{22}}<\frac{1}{r_{n-1}+s_{n-1}}<\frac{1}{m+\frac{4}{3}}.
\]
It now follows easily  that  $\mu''_n(K)<\tfrac{4}{11}=0.363636\dots.$

The same kind of computation shows that in case (b) we have 
\[\tfrac{8}{3}<r_{n-1}<\tfrac{59}{22},\;\;\; 0<s_{n-1}<\tfrac{1}{7},\;\;\;\tfrac{22}{15}<r_{n}<\tfrac{3}{2},\;\;\; \tfrac{7}{15}<s_{n}<\tfrac{1}{2}\]
and so
$\mu''_n(K)<\tfrac{157}{870}=0.18046\dots.$

By  \eqref{mu} and  \eqref{max}   the result follows.
\end{proof}

\bigskip\noindent
 Lemmas \ref{l0} and \ref{l1} prove the first statement of Proposition~\ref{K2}.

\bigskip\noindent
\begin{lemma}\label{l2}
Suppose that  $K$ contains the subsequence $(1,2,1,m)$. If $m \neq 6 $
then $K$ is not exceptional.
\end{lemma}

\begin{proof}

Suppose that $1<m<5.$
By Lemmas \ref{l0} and  \ref{l1} we may assume that  for some $n\in \Z$  
\[r_{n-1}=[2,1,m,1,m',\dots],s_{n-1}=[0,1,m'',\dots],r_n=[1,m,1,m',\dots]\;\mathrm{and}\;s_n=[0,2,1,m'',\dots],\] 
where $m',m''\geq 2$.
It follows that
\[
[2,1,m,1,2]<r_{n-1}<[2,1,m,1]\;\;\mathrm{and}\;\;\;\tfrac{2}{3}<s_{n-1}<1,\]
while
\[
[1,m,1]<r_n<[1,m,1,2]\;\;\;\mathrm{and}\;\;\; [0,2,1]<s_n<[0,2,1,2].\]
Hence
\[
\frac{m+2}{4 m+7}<\frac{1}{r_{n-1}+s_{n-1}}<\frac{9 m+15}{33 m+46}\;\;\mathrm{and}\;\;\frac{8 (3 m+2)}{33 m+46}<\frac{1}{r_{n}+s_{n}}<\frac{3 (m+1)}{4 m+7}.
\]
A calculation now shows that for $m\leq 4$
\[
\left|\frac{1}{r_n+s_n}-\frac{1}{r_{n-1}+s_{n-1}}\right|<\frac{2m+1}{4m+7}<\frac{9}{23}=0.391304\dots
\]
and the statement of the Lemma in this case  follows by \eqref{mup}, \eqref{mu}  and \eqref{max}.

Now assume that $m \geq 7.$
For this we will apply \eqref{muppp} to the reversed sequence $(m, 1,2,1)$.
By the above  we may assume that 
\[ r_{n-1}=[m,1,2,1,m',1,\dots],
s_{n-1}=[0,1,m'',\dots], r_n=[1,2,1,m',\dots], s_n=[0,m,1,m'',\dots]\]
where $ m'\geq 5$ and $m''\geq 2 $.
Then a calculation 
using
\[r_{n-1}>[m,1,2,1,5],\;[1,2,1]<r_n<[1,2,1,5],\;s_{n-1}>[0,1,2],\;[0,m,1]<s_n<[0,m,1,2]\]
gives
\[
0<\frac{2}{r_{n-1}+s_{n-1}}<\frac{2}{m+\frac{97}{69}}\;\;\;\mathrm{and}\;\;\;\frac{3 m-7}{9 m+6}< \frac{r_n-s_n-1}{r_n+s_n}<\frac{3 (6 m-11)}{17 (4 m+7)}.
\]

Thus we have
\begin{align*}
\left|\frac{2}{r_{n-1}+s_{n-1}}+\frac{r_n-s_n-1}{r_n+s_n}\right|<&\frac{3 \left(414 m^2+2951 m+4407\right)}{17 (4 m+7) (69 m+97)} \leq \frac{2721}{6902}=0.394234\dots
\end{align*}
for $m \geq 7.$

Suppose now that $K$ contains $(1,2,1,5)$. 
By the above we may assume that  $K$ contains $(1,5,1,2,1,5,1)$ or $(1,6,1,2,1,5,1)$.

If $K$ contains $(1,5,1,2,1,5,1)$ we have for some $n$ that
\[
[1,2,1,5,1]<r_{n-1}<[1,2,1,5,1,2],\;\;\;\; [0,5,1]<s_{n-1}<[0,5,1,2]
\]
and
\[
[2,1,5,1,2]<r_{n}<[2,1,5,1],\;\;\;\; [0,1,5,1,2]<s_{n}<[0,1,5,1].
\]
This gives
\begin{align*}
0.65473\dots=\tfrac{969}{1480}<\frac{1}{r_{n-1}+s_{n-1}}&<\tfrac{60}{91}=0.659341\dots\;\;\mathrm{and}\\
0.269231\dots=\tfrac{7}{26}<\frac{1}{r_{n}+s_{n}}&<\tfrac{10}{37}=0.27027\dots
\end{align*}
and hence
\[
\mu_n'(K)<\tfrac{ 71}{182}=0.39011\dots.
\]
Similarly, if $K$ contains $(1,5,1,2,1,6,1)$ we have for some $n$ the same inequalities for $s_{n-1}$ and $s_n$
while
\[
[1,2,1,6,1]<r_{n-1}<[1,2,1,6,1,2] \;\;\;\; \mathrm{and}\;\;\;\;[2,1,6,1,2]<r_{n}<[2,1,6,1].
\]
This gives 
\begin{align*}
0.655757\dots=\tfrac{1122}{1711}<\frac{1}{r_{n-1}+s_{n-1}}&<\tfrac{138}{209}=0.669856\dots\;\;\mathrm{and}\\
0.267943\dots=\tfrac{56}{209}<\frac{1}{r_{n}+s_{n}}&<\tfrac{460}{1711}=0.268849\dots.
\end{align*}
This gives
\[
\mu'_n(K) <\tfrac{ 82}{209}=0.392344\dots.
\]
Hence by \eqref{mu} and \eqref{max} we are done.
\end{proof}

If we now assume that $K$  contains $(1,2,1,6)$ then in fact we may assume that $K$ contains 
$(1,2,1,6,1,m,1)$ where $m\geq 2.$

\begin{lemma}\label{l4}
Suppose that $K$ contains $(1,2,1,6,1,m,1)$
where $m>2$.
Then $K$ is not exceptional.
\end{lemma}
\begin{proof}
We may assume that
 for some $n$ 
\[
[6,1,m,1,2]<r_{n-1},\;\;\;\; [0,1,2]<s_{n-1}
\]
and
\[
[1,m,1]<r_{n}<[1,m,1,2],\;\;\;\; [0,6,1,2,1]<s_{n}<[0,6,1,2].
\]
Thus we have that
\[
0<\frac{2}{r_{n-1}+s_{n-1}}<\frac{6 (3 m+5)}{69 m+106}\]
and
\[
\frac{-9 m^2+45 m+34}{(m+1) (69 m+106)}< \frac{r_n-s_n-1}{r_n+s_n}<-\frac{(m+1) (12 m-73)}{(3 m+2) (31 m+58)}.
\]
Hence
\begin{align*}
\left|\frac{2}{r_{n-1}+s_{n-1}}+\frac{r_n-s_n-1}{r_n+s_n}\right|<\frac{846 m^3+9975 m^2+20671 m+11218}{(3 m+2) (31 m+58) (69 m+106)}\\ \leq \frac{185848}{519893}=0.357474\dots.
\end{align*}
for $m \geq 3.$
\end{proof}
This gives the second statement and thus completes the proof of Proposition~\ref{K2}.

\bigskip\noindent
Suppose now that $K=(\dots, 1, m_1,1,m_2,1,\dots)$ is exceptional with $K \neq K_1$ and $K \neq K_2$.  By Proposition~\ref{K2}
and \eqref{K} we have that $m_j\geq 3$ for all $j$ with at least one $m_j> 3$.
Now for $K_3$ given in \eqref{K} we have that
\[
\mu_n'''(K_3)= \left|\frac{2}{r_{n-1}+s_{n-1}}+\frac{r_n-s_n-1}{r_n+s_n}\right|=\frac{1}{4} \left(\sqrt{21}-3\right),
\]
when
\[
r_{n}=[1,3,1,3,1,\dots]\;\;\;\mathrm{and}\;\;\; s_{n}=[0,4,1,3,1,3,\dots].
\]
Since $\mu(K)\leq \mu_n'''(K)$  by \eqref{mu} and  \eqref{max},
Proposition~\ref{mp} is  a consequence of the following lemma, which implies that \[\mu_n'''(K) < \mu_n'''(K_3)\]
unless $K$ is equivalent to $K_3.$

\begin{lemma}
Suppose that
\[
r_n=[1,m_1,1,m_2,1,\dots], \;\;\;s_n=[0,m_0,1,m_{-1},1,\dots]
\]
and 
\[
r'_n=[1,m'_1,1,m'_2,1,\dots],\;\;\; s'_n=[0,m'_0,1,m'_{-1},1,\dots]
\]
with $m'_j \geq m_j\geq 1$ for all $j \in \Z.$
Then
\[
\left|\frac{2}{r'_{n-1}+s'_{n-1}}+\frac{r'_n-s'_n-1}{r'_n+s'_n}\right|\leq \left|\frac{2}{r_{n-1}+s_{n-1}}+\frac{r_n-s_n-1}{r_n+s_n}\right|,
\]
with equality if and only if $ m'_j= m_j$ for all $j \in \Z.$
\end{lemma}
\begin{proof}
Under our assumptions we have 
 $r_{n-1}\leq r'_{n-1}$ and $r_{n}\geq r'_{n}$.
 Now $s_{n-1}=\frac{1}{s_n}-m_0$ and $r_{n-1}=\frac{1}{r_n}+m_0$ and similarly for $r',s'$.
 Thus 
 \[\frac{2}{r_{n-1}+s_{n-1}}+\frac{r_n-s_n-1}{r_n+s_n}=\frac{2r_ns_n+r_n-s_n-1}{r_n+s_n}.\]
 The result now follows since
 the function
  \[
F(x,y)=\frac{2 x y+x-y-1}{x+y}
\]
satisfies $0<F(x_1,y_1)\leq F(x_2,y_2)$ whenever $1<x_1\leq x_2<2$ and $0<y_1\leq y_2<1$,
with equality if and only if $x_1= x_2$ and $y_1=y_2$.
To see this, use that the gradient of $F(x,y)$ is given by
\[
\nabla F(x,y) =\left(\tfrac{2 y (y+1)+1}{(x+y)^2},\tfrac{2 (x-1) x+1}{(x+y)^2}\right). \qedhere
\]
\end{proof}

\section{Approximating $\mathcal C_3$}\label{final}

We conclude by justifying the final statement of Theorem~\ref{main}.
For $\ell\geq 1$, let
\[
K_{\ell}=(\overline{3,1},4,1,3,1,\dots, 3,1,4,\overline{1,3}), \qquad k_0 = k_{2\ell} = 4,
\]
where the number of  1's between the 4's is given by $\ell$. 
Propositions~\ref{finl} and \ref{mp} show that the billiard $\mathcal B_\ell$ associated to $K_\ell$ satisfies $\lambda_i(\mathcal B_\ell) > \tfrac{1}{3} (3+\sqrt{21})$.

\begin{proposition}\label{dec}
Let $\mathcal{B}_\ell$ be the billiard associated to the class of $K_{\ell}$ by Theorem \ref{map2}. Then
\[\lim_{\ell \rightarrow \infty}\lambda_i(\mathcal{B}_\ell)=\tfrac{1}{3} (3+\sqrt{21}).\]
\end{proposition}
\begin{proof}
For a fixed $\ell$, define $r_n=r_n(\ell)$ and $s_n=s_n(\ell)$ as in \eqref{rs} for the sequence $K_\ell$.
Then
\begin{equation*}
  r_1 = [1,3,1,\ldots,3,1,4,\overline{1,3}] \quad \text{ and } \quad s_1 = [0,4,\overline{1,3}] = \tfrac 12(5-\sqrt {21}),
\end{equation*}
where the bar indicates a repeated sequence, and the number of 1's before the 4 in $r_1$ is given by $\ell$.
One can easily compute an explicit formula for $r_1$ using the recurrence relation
\begin{equation}  \label{eq:r1-recu}
  r_1(\ell+1) = 1 + \frac{1}{3 + \frac 1{r_1(\ell)}}, \qquad r_1(1) = \tfrac 12(7-\sqrt{21}).
\end{equation}
Let
$\ep = \frac 12(5+\sqrt{21})$ denote the fundamental unit in $\Q(\sqrt{21})$.
Then by \eqref{muppp} and \eqref{recu} we have the formula
\begin{equation*}
  \mu'''_1(K_\ell) = \frac{(-3+2\sqrt{21})\ep^\ell - 3\ep^{1-\ell}}{(11+\sqrt{21})\ep^\ell - \frac 12\ep^{-\ell}}.
\end{equation*}
We claim that $\mu(K_\ell) = \mu_1'''(K_\ell)$; the result then follows easily since $\lambda_i(\mathcal B_\ell) = \mu(K_\ell)^{-1}$.

It is straightforward to check that $\mu(K_\ell) = \mu'''(K_\ell)$ for small $\ell$, so we assume $\ell\geq 3$.
To prove that $\mu(K_\ell) = \mu'''(K_\ell)$, we first show that $\mu_n'(K_\ell),\mu_n''(K_\ell)>\frac{-3+\sqrt{21}}{4}=.395\ldots$ for all $n$ and that $\mu_n'''(K_\ell)>\frac{-3+\sqrt{21}}{4}$ if $n\neq \pm1, 2\ell\pm 1$.
If $n$ is odd, then
  \begin{equation*}
    \tfrac{1+\sqrt 2}{2} = [\overline{1,4}] \leq r_n \leq [\overline{1,3}] = \tfrac{3+\sqrt{21}}{6}, \quad \tfrac{-1+\sqrt 2}{2} = [0,\overline{4,1}] \leq s_n \leq [0,\overline{3,1}] = \tfrac{-3+\sqrt{21}}{6}.
  \end{equation*}
If $n$ is even and $k_n=4$, then
\begin{equation*}
  \tfrac{5+\sqrt{21}}{2} = [4,1,\overline{3,1}] \leq r_n \leq [\overline{4,1}] = 2+2\sqrt{2}, \quad
  \tfrac{-3+\sqrt{21}}{2} = [0,\overline{1,3}] \leq s_n \leq [0,1,3,\overline{1,4}] = \tfrac{3-\sqrt 2}{2},
\end{equation*}
while if $n$ is even and $k_n\neq 4$, then
\begin{equation*}
  \tfrac {3+\sqrt {21}}{2} = [\overline{3,1}] \leq r_n \leq [3,1,\overline{4,1}] = 1+2\sqrt 2, \quad
  \tfrac{-3+\sqrt{21}}{2} = [0,\overline{1,3}] \leq s_n \leq [0,\overline{1,4}] = -2+2\sqrt 2,
\end{equation*}
It follows from \eqref{mup}, \eqref{mupp}, and \eqref{recu} that in each case, $\mu_n'(K_\ell), \mu_n''(K_\ell) > .399$.
Similarly, $\mu_n'''(K_\ell)>.399$ if $n$ is even.
If $n$ is odd and $k_{n+1},k_{n-1}\neq 4$ then
\begin{equation*}
  r_n \geq [1,3,\overline{1,4}] = \tfrac{6+2\sqrt 2}{7}, \qquad s_n \geq [0,3,1,\overline{4,1}] = \tfrac{-1+2\sqrt 2}{7},
\end{equation*}
from which it follows that $\mu_n'''(K_\ell)>.399$.

It remains to show that $\mu_{-1}(K_\ell) > \mu_1(K_\ell)$ since, by symmetry, we have $\mu_{2\ell+1}(K_\ell) > \mu_{2\ell-1}(K_\ell) = \mu_1(K_\ell)$.
By \eqref{recu} we have
\begin{equation}
  \mu'''_{-1}(K_\ell) = \mu'''_1(K_\ell) + \frac{6(4r_1s_1+s_1-r_1)}{r_1+s_1}.
\end{equation}
Thus $\mu_{-1}(K_\ell) > \mu_1(K_\ell)$ if and only if
$r_1 < \frac{s_1}{1-4s_1} = \frac 16(3+\sqrt{21})$.
This inequality follows from the relation \eqref{eq:r1-recu}, which completes the proof.
\end{proof}

\end{document}